\documentclass[reqno,11pt]{amsart}

\usepackage{amsmath}
\usepackage{amsthm}
\usepackage{amsfonts}
\usepackage{amssymb}
\usepackage{enumerate}
\usepackage{graphicx}
\usepackage{color}
\usepackage[usenames,dvipsnames]{xcolor}
\usepackage{verbatim}

\newcommand{\oindicator}[1]{\ensuremath{\mathbf{1}_{{#1}}}}
\newcommand\abs[1]{\left|#1\right|}
\newcommand\norm[1]{\left\|#1\right\|}
\newcommand\JOmega{\Psi}

\numberwithin{equation}{section}

\DeclareMathOperator{\tr}{tr}

\DeclareMathOperator{\diag}{diag}
\DeclareMathOperator{\dist}{dist}
\DeclareMathOperator{\rank}{rank}

\newcommand{\transp}[1]{#1^{\mathrm T}}

\newcommand{\Prob}{\mathbb{P}}
\newcommand{\E}{\mathbb{E}}

\renewcommand\Re{\operatorname{Re}}
\renewcommand\Im{\operatorname{Im}}
\newcommand{\eps}{\varepsilon}

\newcommand{\mat}{\mathbf}

\newcommand{\Esc}{\mathcal{E}^{\mathrm{sc}}}
\newcommand{\Emp}{\mathcal{E}^{\mathrm{MP}}}
\newcommand{\EQ}{\mathcal{E}^{Q}}

\newcommand{\pfrac}[2]{\left(\frac{#1}{#2}\right)}

\theoremstyle{plain}
  \newtheorem{theorem}{Theorem}[section]

  \newtheorem{lemma}[theorem]{Lemma}
  \newtheorem{corollary}[theorem]{Corollary}
  
  \newtheorem{question}[theorem]{Question}

\theoremstyle{definition}
  \newtheorem{definition}[theorem]{Definition}
  \newtheorem{example}[theorem]{Example}
  \newtheorem{remark}[theorem]{Remark}

\begin{document}
\title{Spectra of nearly Hermitian random matrices}

\author[S. O'Rourke]{Sean O'Rourke}
\address{Department of Mathematics, University of Colorado at Boulder, Boulder, CO 80309  }
\email{sean.d.orourke@colorado.edu}

\author[P. Wood]{Philip Matchett Wood}
\thanks{The second author was partially supported by National Security Agency (NSA) Young Investigator Grant number H98230-14-1-0149.} 
\address{Department of Mathematics, University of Wisconsin-Madison, 480 Lincoln Dr., Madison, WI 53706 }
\email{pmwood@math.wisc.edu}

\begin{abstract}
We consider the eigenvalues and eigenvectors of matrices of the form $\mat M + \mat P$, where $\mat M$ is an $n \times n$ Wigner random matrix and $\mat P$ is an arbitrary $n \times n$ deterministic matrix with low rank.  In general, we show that none of the eigenvalues of $\mat M + \mat P$ need be real, even when $\mat P$ has rank one. We also show that, except for a few outlier eigenvalues, most of the eigenvalues of $\mat M + \mat P$ are within $n^{-1}$ of the real line, up to small order corrections.  We also prove a new result quantifying the outlier eigenvalues for multiplicative perturbations of the form $\mat S (\mat I + \mat P)$, where $\mat S$ is a sample covariance matrix and $\mat I$ is the identity matrix.  We extend our result showing all eigenvalues except the outliers are close to the real line to this case as well.  As an application, we study the critical points of the characteristic polynomials of nearly Hermitian random matrices.  
\end{abstract}

\maketitle

\section{Introduction and notation}

The fundamental question of perturbation theory is the following.  How does a function change when its argument is subject to a perturbation?  In particular, matrix perturbation theory is concerned with perturbations of matrix functions, such as solutions to linear system, eigenvalues, and eigenvectors.  The purpose of this paper is to analyze the eigenvalues and eigenvectors of large perturbed random matrices.  

For an $n \times n$ matrix $\mat M$ with complex entries, the eigenvalues of $\mat M$ are the roots in $\mathbb{C}$ of the characteristic polynomial $p_{\mat M}(z) := \det(\mat M - z \mat{I})$, where $\mat{I}$ denotes the identity matrix.  Let $\lambda_1(\mat{M}), \ldots, \lambda_n(\mat M)$ denote the eigenvalues of $\mat M$ counted with multiplicity, and let $\Lambda(\mat M)$ denote the set of all eigenvalues of $\mat M$.  In particular, if $\mat{M}$ is Hermitian (that is, $\mat M = \mat M^\ast$), then $\Lambda(\mat M) \subset \mathbb{R}$.  

The central question in the perturbation theory of eigenvalues is the following.  Given a matrix $\mat M$ and a perturbation $\mat P$ of $\mat M$, how are the spectra $\Lambda(\mat M)$ and $\Lambda(\mat M + \mat P)$ related?  Many challenges can arise when trying to answer this question.  For instance, different classes of matrices behave differently under perturbation.  

The simplest case to consider is when both $\mat M$ and $\mat P$ are Hermitian.  Indeed, if $\mat M$ and $\mat P$ are Hermitian $n \times n$ matrices then the eigenvalues of $\mat M$ and $\mat M + \mat P$ are real and the classic Hoffman--Wielandt theorem \cite{HW} states that there exists a permutation $\pi$ of $\{1, \ldots, n\}$ such that
\begin{equation} \label{eq:hw}
	\sum_{j=1}^n \left| \lambda_{\pi(j)}(\mat{M}) - \lambda_j(\mat{M} + \mat{P}) \right|^2 \leq \|\mat{P}\|^2_2.
\end{equation}
Here, $\| \mat{P} \|_2$ denotes the Frobenius norm of $\mat P$ defined by the formula
\begin{equation} \label{eq:frobenius}
	\| \mat P \|_2 := \sqrt{ \tr (\mat P \mat P^\ast ) } = \sqrt{ \tr (\mat P^\ast \mat P ) }.
\end{equation}
The classic Hoffman--Weilandt theorem \cite{HW} is in fact a bit more general: \eqref{eq:hw} holds assuming only that $\mat M$ and $\mat M + \mat P$ are normal.

However, in this note, we focus on the case when $\mat M$ is Hermitian but $\mat P$ is arbitrary.  In this case, the spectrum $\Lambda(\mat M)$ is real, but the eigenvalues of $\mat M + \mat P$ need not be real.  Indeed, in the following example, we show that even when $\mat P$ contains only one nonzero entry, all the eigenvalues of $\mat M + \mat P$ can be complex.  

\begin{example} \label{ex:toeplitz}
Consider an $n \times n$ tridiagonal Toeplitz matrix $\mat{T}$ with zeros on the diagonal and ones on the super- and sub-diagonals.  The matrix $\mat T$ has eigenvalues $-2\cos\left(\frac{k\pi}{n+1}\right)$ for $k=1,2,\dots,n$ (see for example \cite{KST1999}).  Let $p_\mat{T}$ denote the characteristic polynomial for $\mat{T}$.  Let $\mat{P}$ be the matrix with every entry zero except for the $(1,n)$-entry which is set to  $\sqrt{-1}$, the imaginary unit.\footnote{Here and in the sequel, we use $\sqrt{-1}$ to denote the imaginary unit and reserve $i$ as an index.}  Then the characteristic polynomial of $\mat{T} + \mat P$ is $p_{\mat{T}}(z) + \sqrt{-1}$.  Since $p_{\mat{T}}(x) \in \mathbb{R}$ for all $x \in \mathbb{R}$, it follows that none of the eigenvalues of $\mat{T} + \mat{P}$ are real.  
\end{example}

Example \ref{ex:toeplitz} shows that we cannot guarantee that even one of the eigenvalues of $\mat M + \mat P$ is real.  However, this example does raise the following question.
\begin{question} \label{q:main}
If $\mat{M}$ is Hermitian and $\mat P$ is arbitrary, how far are the eigenvalues of $\mat M + \mat P$ from the real line?
\end{question}

Question \ref{q:main} was addressed by Kahan \cite{K}.  We summarize Kahan's results below in Theorem \ref{thm:kahan}, but we first fix some notation.  For an $n \times n$ matrix $\mat M$, let $\| \mat M \|$ denote the spectral norm of $\mat M$.  That is, $\| \mat M \|$ is the largest singular value of $\mat M$.  We also denote the real and imaginary parts of $\mat M$ as 
$$ \Re(\mat M) := \frac{\mat M + \mat M^\ast}{2}, \qquad \Im(\mat M) := \frac{ \mat M - \mat M^\ast }{2i}. $$
It follows that $\mat M$ is Hermitian if and only if $\Im(\mat M) = 0$.  

\begin{theorem}[Kahan \cite{K}] \label{thm:kahan}
Let $\mat M$ be an $n \times n$ Hermitian matrix with eigenvalues $\lambda_1 \geq \cdots \geq \lambda_n$.  Let $\mat P$ be an arbitrary $n \times n$ matrix, and denote the eigenvalues of $\mat M + \mat P$ as $\mu_1 + \sqrt{-1} \nu_1, \ldots, \mu_n + \sqrt{-1} \nu_n$, where $\mu_1 \geq \cdots \geq \mu_n$.  Then
\begin{equation} \label{eq:kahan1}
	\sup_{1 \leq k \leq n} |\nu_k| \leq \| \Im(\mat P) \|, \qquad \sum_{k=1}^n \nu_k^2 \leq \| \Im(\mat P) \|_2^2, 
\end{equation}
and
\begin{equation} \label{eq:kahan2}
	\sum_{k=1}^n \left| (\mu_k + \sqrt{-1} \nu_k) - \lambda_k \right|^2 \leq 2 \| \mat P \|_2^2. 
\end{equation}
\end{theorem}

\begin{remark}
In the case that $\mat{M}$ is only assumed to be normal (instead of Hermitian) but $\mat M + \mat{P}$ is arbitrary, Sun \cite{S} proved that there exists a permutation $\pi$ of $\{1, \ldots, n\}$ such that
\begin{equation} \label{eq:sun}
	\sum_{j=1}^n \left| \lambda_{\pi(j)}(\mat{M}) - \lambda_j(\mat{M} + \mat{P}) \right|^2 \leq n\|\mat{P}\|^2_2, 
\end{equation}
and \cite{S} also shows that this bound is sharp.
\end{remark}

We refer the reader to \cite[Section IV.5.1]{SS} for a discussion of Kahan's results as well as a concise proof of Theorem \ref{thm:kahan}.  The bounds in \eqref{eq:kahan1} were shown to be sharp in \cite{K}.  In the next example, we show that \eqref{eq:kahan2} is also sharp.  

\begin{example} \label{ex:sharp}
Consider the matrices $\mat M = \begin{bmatrix} 0 & 1\\ 1 & 0 \end{bmatrix}$ and $\mat P = \begin{bmatrix} 0 & 0 \\ -1 & 0\end{bmatrix}$.  Then $\mat M$ has eigenvalues $\pm 1$ and $\mat M + \mat P$ has only the eigenvalue $0$ with multiplicity two.  Since $\| \mat P \|_2 = 1$, it follows that
$$ |1-0|^2 + |-1 - 0|^2 = 2 \| \mat P \|_2^2, $$
and hence the bound in \eqref{eq:kahan2} is sharp.  
\end{example}

In this note, we address Question \ref{q:main} when $\mat{M}$ is a Hermitian random matrix and $\mat P$ is a deterministic, low rank perturbation.  In particular, our main results (presented in Section \ref{sec:main}) show that, in this setting, one can improve upon the results in Theorem \ref{thm:kahan}.  One might not expect this improvement since the bounds in Theorem \ref{thm:kahan} are sharp; however, the bounds appear to be sharp for very contrived examples (such as Example \ref{ex:sharp}).  Intuitively, if we consider a random matrix $\mat M$, we expect with high probability to avoid these worst-case scenarios, and thus, some improvement is expected.  Before we present our main results, we describe the ensembles of Hermitian random matrices we will be interested in.

\subsection{Random matrix models}
We consider two important ensembles of Hermitian random matrices.  The first ensemble was originally introduced by Wigner \cite{W} in the 1950s to model Hamiltonians of atomic nuclei.

\begin{definition}[Wigner matrix]
Let $\xi, \zeta$ be real random variables.  We say $\mat{W}$ is a \emph{real symmetric Wigner matrix} of size $n$ with atom variables $\xi$ and $\zeta$ if $\mat{W} = (w_{ij})_{i,j=1}^n$ is a random real symmetric $n \times n$ matrix that satisfies the following conditions.
\begin{itemize}
\item $\{w_{ij} : 1 \leq i \leq j \leq n\}$ is a collection of independent random variables.
\item $\{w_{ij} : 1 \leq i < j \leq n\}$ is a collection of independent and identically distributed (iid) copies of $\xi$.
\item $\{w_{ii} : 1 \leq i \leq n\}$ is a collection of iid copies of $\zeta$.  
\end{itemize}
\end{definition}

\begin{remark}
One can similarly define Hermitian Wigner matrices where $\xi$ is allowed to be a complex-valued random variable.  However, for simplicity, we only focus on the real symmetric model in this note.  
\end{remark}

The prototypical example of a Wigner real symmetric matrix is the \emph{Gaussian orthogonal ensemble} (GOE).  The GOE is defined as an $n \times n$ Wigner matrix with atom variables $\xi$ and $\zeta$, where $\xi$ is a standard Gaussian random variable and $\zeta$ is a Gaussian random variable with mean zero and variance two.  Equivalently, the GOE can be viewed as the probability distribution  
\begin{equation*}
	\Prob(d \mat W) = \frac{1}{Z_n} \exp\left({-\frac{1}{4}\tr{\mat{W}^2}}\right) d \mat W
\end{equation*}
on the space of $n \times n$ real symmetric matrices, where $d \mat W$ refers to the Lebesgue measure on the $n(n+1)/2$ different elements of the matrix.  Here $Z_n$ denotes the normalization constant.

We will also consider an ensemble of sample covariance matrices.  

\begin{definition}[Sample covariance matrix]
Let $\xi$ be a real random variable.  We say $\mat{S}$ is a \emph{sample covariance matrix} with atom variable $\xi$ and parameters $(m,n)$ if $\mat S = \mat X^\mathrm{T} \mat X$, where $\mat X$ is a $m \times n$ random matrix whose entries are iid copies of $\xi$.  
\end{definition}

A fundamental result for Wigner random matrices is Wigner's semicircle law, which describes the global behavior of the eigenvalues of a Wigner random matrix.  Before stating the result, we present some additional notation.  For an arbitrary $n \times n$ matrix $\mat A$, we define the \emph{empirical spectral measure} $\mu_{\mat A}$ of $\mat A$ as 
$$ \mu_{\mat A} := \frac{1}{n} \sum_{k=1}^n \delta_{\lambda_k(\mat A)}. $$
In general, $\mu_{\mat A}$ is a probability measure on $\mathbb{C}$, but if $\mat A$ is Hermitian, then $\mu_{\mat A}$ is a probability measure on $\mathbb{R}$.  In particular, if $\mat A$ is a random matrix, then $\mu_{\mat A}$ is a random probability measure on $\mathbb{C}$.  Let us also recall what it means for a sequence of random probability measures to converge weakly. 

\begin{definition}[Weak convergence of random probability measures] \label{def:weakconvergence} 
Let $T$ be a topological space (such as $\mathbb{R}$ or $\mathbb{C}$), and let $\mathcal{B}$ be its Borel $\sigma$-field.  Let $(\mu_n)_{n\geq 1}$ be a sequence of random probability measures on $(T,\mathcal{B})$, and let $\mu$ be a probability measure on $(T,\mathcal{B})$.  We say \emph{$\mu_n$ converges weakly to $\mu$ in probability} as $n \to \infty$ (and write $\mu_n \to \mu$ in probability) if for all bounded continuous $f:T \to \mathbb{R}$ and any $\eps > 0$,
$$ \lim_{n \to \infty} \Prob \left( \left| \int f d\mu_n - \int f d\mu \right| > \eps \right) = 0. $$
In other words, $\mu_n \to \mu$ in probability as $n \to \infty$ if and only if $\int f d\mu_n \to \int f d\mu$ in probability for all bounded continuous $f: T \to \mathbb{R}$.  Similarly, we say \emph{$\mu_n$ converges weakly to $\mu$ almost surely} (or, equivalently, \emph{$\mu_n$ converges weakly to $\mu$ with probability $1$}) as $n \to \infty$ (and write $\mu_n \to \mu$ almost surely) if for all bounded continuous $f:T \to \mathbb{R}$,
$$ \lim_{n \to \infty} \int f d\mu_n = \int f d\mu $$
almost surely.    
\end{definition}

Recall that Wigner's semicircle law is the (non-random) measure $\mu_{\mathrm{sc}}$ on $\mathbb{R}$ with density 
\begin{equation} \label{eq:def:rhosc}
	\rho_{\mathrm{sc}}(x) := \left\{
	\begin{array}{ll}
		\frac{1}{2 \pi} \sqrt{4 - x^2}, & \text{if } |x| \leq 2, \\
		0, & \text{otherwise}.
	\end{array} \right. 
\end{equation}

\begin{theorem}[Wigner's semicircle law; Theorem 2.5 from \cite{BSbook}] \label{thm:wigner}
Let $\xi$ and $\zeta$ be real random variables; assume $\xi$ has unit variance.  For each $n \geq 1$, let $\mat W_n$ be an $n \times n$ real symmetric Wigner matrix with atom variables $\xi$ and $\zeta$.  Then, with probability $1$, the empirical spectral measure $\mu_{\frac{1}{\sqrt{n}}\mat {W}_n}$ of $\frac{1}{\sqrt{n}} \mat W_n$ converges weakly on $\mathbb{R}$ as $n \to \infty$ to the (non-random) measure $\mu_{\mathrm{sc}}$ with density given by \eqref{eq:def:rhosc}.  
\end{theorem}

For sample covariance matrices, the Marchenko--Pastur law describes the limiting global behavior of the eigenvalues.  Recall that the Marchenko--Pastur law is the (non-random) measure $\mu_{\mathrm{MP},y}$ with parameter $y > 0$ which has density
\begin{equation} \label{eq:def:rhoMP}
	\rho_{\mathrm{MP},y}(x) := \left\{
	\begin{array}{ll}
		\frac{\sqrt{y}}{2 x\pi} \sqrt{ (x - \lambda_-)(\lambda_+ - x) }, & \text{if } \lambda_- \leq x \leq \lambda_+, \\
		0, & \text{otherwise}
	\end{array} \right. 
\end{equation}
and with point mass $1 - y$ at the origin if $y < 1$, where 
\begin{equation} \label{eq:def:lambdapm}
	\lambda_{\pm} := \sqrt{y} \left( 1 \pm \frac{1}{\sqrt{y}} \right)^2. 
\end{equation}

\begin{theorem}[Marchenko--Pastur law; Theorem 3.6 from \cite{BSbook}] \label{thm:mp}
Let $\xi$ be a real random variable with mean zero and unit variance.  For each $n \geq 1$, let $\mat S_n$ be a sample covariance matrix with atom variable $\xi$ and parameters $(m,n)$, where $m$ is a function of $n$ such that $y_n := \frac{m}{n} \to y \in (0, \infty)$ as $n \to \infty$.  Then, with probability $1$, the empirical spectral measure $\mu_{\frac{1}{\sqrt{mn}} \mat S_n}$ of $\frac{1}{\sqrt{mn}} \mat S_n$ converges weakly on $\mathbb{R}$ as $n \to \infty$ to the (non-random) measure $\mu_{\mathrm{MP},y}$.  
\end{theorem}

\subsection{Notation} 
We use asymptotic notation (such as $O,o$) under the assumption that $n \rightarrow \infty$.  We use $X = O(Y)$ to denote the bound $X \leq CY$ for all sufficiently large $n$ and for some constant $C$.  Notation such as $X=O_k(Y)$ mean that the hidden constant $C$ depends on another constant $k$.  The notation $X=o(Y)$ or $Y=\omega(X)$ means that $X/Y \rightarrow 0$ as $N \rightarrow \infty$.  

For an event $E$, we let $\oindicator{E}$ denote the indicator function of $E$, and $E^{c}$ denotes the complement of $E$.  We write a.s. and a.e. for almost surely and Lebesgue almost everywhere, respectively. We let $\sqrt{-1}$ denote the imaginary unit and reserve $i$ as an index.  

For any matrix $\mat A$, we define the Frobenius norm $\|\mat A\|_2$ of $\mat A$ by \eqref{eq:frobenius}, and we use $\|\mat{A}\|$ to denote the spectral norm of $\mat{A}$.  We let $\mat{I}_n$ denote the $n \times n$ identity matrix.  Often we will just write $\mat{I}$ for the identity matrix when the size can be deduced from the context.     

We let $C$ and $K$ denote constants that are non-random and may take on different values from one appearance to the next.  The notation $K_p$ means that the constant $K$ depends on another parameter $p$.

\section{Main results} \label{sec:main}

Studying the eigenvalues of deterministic perturbations of random matrices has generated much interest.  In particular, recent results have shown that adding a low-rank  perturbation to a large random matrix barely changes the global behavior of the eigenvalues.  However, as illustrated below, some of the eigenvalues can deviate away from the bulk of the spectrum.  This behavior, sometimes referred to as the BBP transition, was first studied by Johnstone \cite{J} and Baik, Ben Arous, and P\'{e}ch\'{e} \cite{BBP} for spiked covariance matrices.  Similar results have been obtained in \cite{BGGM, BGGM2, BGN, BGN2, CDMF, CDMF2, CDMF3, KY, KY2, P, PRS, RS2} for other Hermitian random matrix models.  Non-Hermitian models have also been studied, including \cite{BC, Raj, Tout, Touterr} (iid random matrices), \cite{OR} (elliptic random matrices), and \cite{BGR, BGR2} (matrices from the single ring theorem).  

In this note, we focus on Hermitian random matrix ensembles perturbed by non-Hermitian matrices.  This model has recently been explored in \cite{OR, R}.  However, the results in \cite{OR, R} address only the ``outlier'' eigenvalues, and leave Question \ref{q:main} unanswered for the bulk of the eigenvalues.  The goal of this paper is to address these bulk eigenvalues.  We begin with some examples.  

\subsection{Some example perturbations}
In Example \ref{ex:toeplitz}, we gave a deterministic example when $\mat{M}$ is Hermitian,  $\mat{P}$ is non-Hermitian, and none of the eigenvalues of $\mat{M} + \mat P$ are real.  We begin this section by giving some random examples where the same phenomenon holds.  We first consider a Wigner matrix perturbed by a diagonal matrix.  

\begin{theorem} \label{thm:nonreal}
For $\mu > 0$, let $\xi$ be a real random variable satisfying
\begin{equation} \label{eq:nondeg}
	\Prob(\xi = x) \leq 1 - \mu \text{ for all } x \in \mathbb{R}, 
\end{equation}
and let $\zeta$ be an arbitrary real random variable.  
Suppose $\mat W$ is an $n \times n$ Wigner matrix with atom variables $\xi$ and $\zeta$.  Let $\mat P$ be the diagonal matrix $\mat P = \diag(0, \ldots, 0, \gamma \sqrt{-1})$ for some $\gamma \in \mathbb{R}$ with $\gamma \neq 0$.  Then, for any $\alpha > 0$, there exists $C > 0$ (depending on $\alpha$ and $\mu$) such that the following holds with probability at least $1 - Cn^{-\alpha}$:
\begin{itemize}
\item if $\gamma > 0$, then all eigenvalues of $\frac{1}{\sqrt{n}} \mat{W} + \mat P$ are in the upper half-plane $\mathbb{C}^+ := \{ z \in \mathbb{C} : \Im(z) > 0 \}$.  
\item if $\gamma < 0$, then all eigenvalues of $\frac{1}{\sqrt{n}} \mat{W} + \mat P$ are in the lower half-plane $\mathbb{C}^- := \{ z \in \mathbb{C} : \Im(z) < 0 \}$.  
\end{itemize}
Moreover, if $\xi$ and $\zeta$ are absolutely continuous random variables, then the above holds with probability $1$.  
\end{theorem}

\begin{remark}
The choice for the last coordinates of $\mat{P}$ to take the value $\gamma \sqrt{-1}$ is completely arbitrary.  Since $\mat W$ is invariant under conjugation by a permutation matrix, the same result also holds for $\mat{P} = \gamma \sqrt{-1} v v^\ast$, where $v$ is any $n$-dimensional standard basis vector.  
\end{remark}

Figure~\ref{fig:Gau+i} depicts a numerical simulation of Theorem \ref{thm:nonreal} when the entries of $\mat W$ are Gaussian.  The proof of Theorem \ref{thm:nonreal} relies on some recent results due to Tao and Vu \cite{TVsimple} and Nguyen, Tao, and Vu \cite{NTV} concerning gaps between eigenvalues of Wigner matrices.  

\begin{figure}
\includegraphics[width=\textwidth]{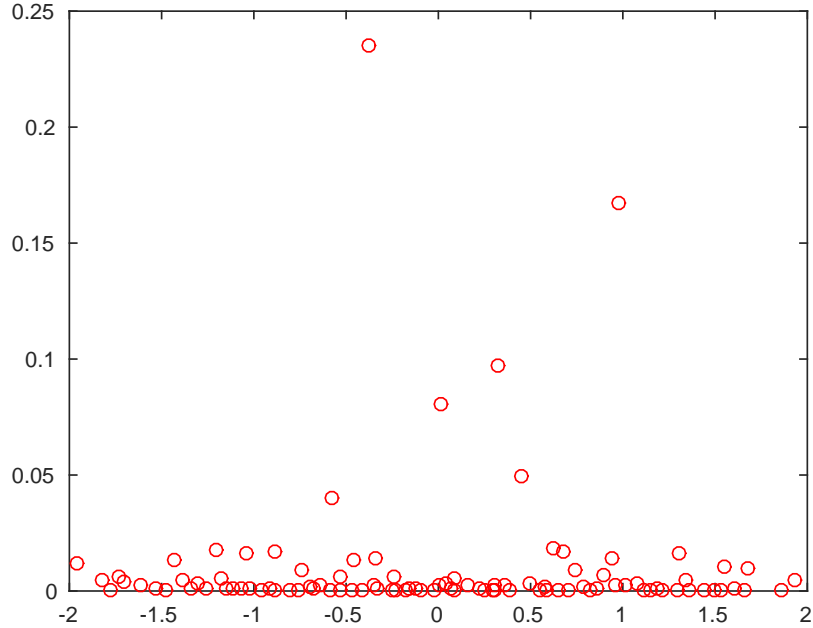}
\caption{ {\bf All imaginary eigenvalues after a rank-1 perturbation.}
Above is a plot of the eigenvalues of $\frac1{ \sqrt n} \mat{W} + \mat P$, where $\mat W$ is a 100 by 100 
GOE matrix,
and where $\mat P = \diag(0,\dots,0,\sqrt{-1})$.  As described in Theorem~\ref{thm:nonreal}, all the eigenvalues have positive imaginary part.  
Here, the minimum imaginary part is $8.6294\times 10^{-7}$. 
}
\label{fig:Gau+i}
\end{figure}

The next result is similar to Theorem~\ref{thm:nonreal} and applies to perturbations of random sample covariance matrices.

\begin{theorem} \label{thm:sampcov-nonreal}
Let $\xi$ be an absolutely continuous real random variable.  Let $\mat S_n$ be a sample covariance matrix with atom variable $\xi$ and parameters $(m,n)$, where $m$ and $n$ are positive integers, and take $r:=\min\{m,n\}$.  Let $v$ be any $n$-dimensional standard basis vector, so one coordinate of $v$ equals 1 and the rest are 0. Then, with probability $1$, 
\begin{itemize}
\item if $\gamma>0$, then $\mat S_n(\mat I + \gamma \sqrt{-1} v v^*) $ has $r$ eigenvalues with positive imaginary part, and 
\item if $\gamma <0$, then $\mat S_n(\mat I + \gamma \sqrt{-1}v v^*) $  has $r$ eigenvalues with negative imaginary part.
\end{itemize}
The remaining $n-r$ eigenvalues of $\mat S_n$ (if any) are all equal to 0.
\end{theorem}

Figure~\ref{fig:SampCov+ivvt} gives a numerical demonstration of Theorem~\ref{thm:sampcov-nonreal}.  We conjecture that Theorem \ref{thm:sampcov-nonreal} can also be extended to the case when $\xi$ is a discrete random variable which satisfies a non-degeneracy condition, such as \eqref{eq:nondeg}.  In order to prove such a result, one would need to extend the results of \cite{NTV, TVsimple} to the sample covariance case; we do not pursue this matter here.  

\begin{figure}
\includegraphics[width=\textwidth]{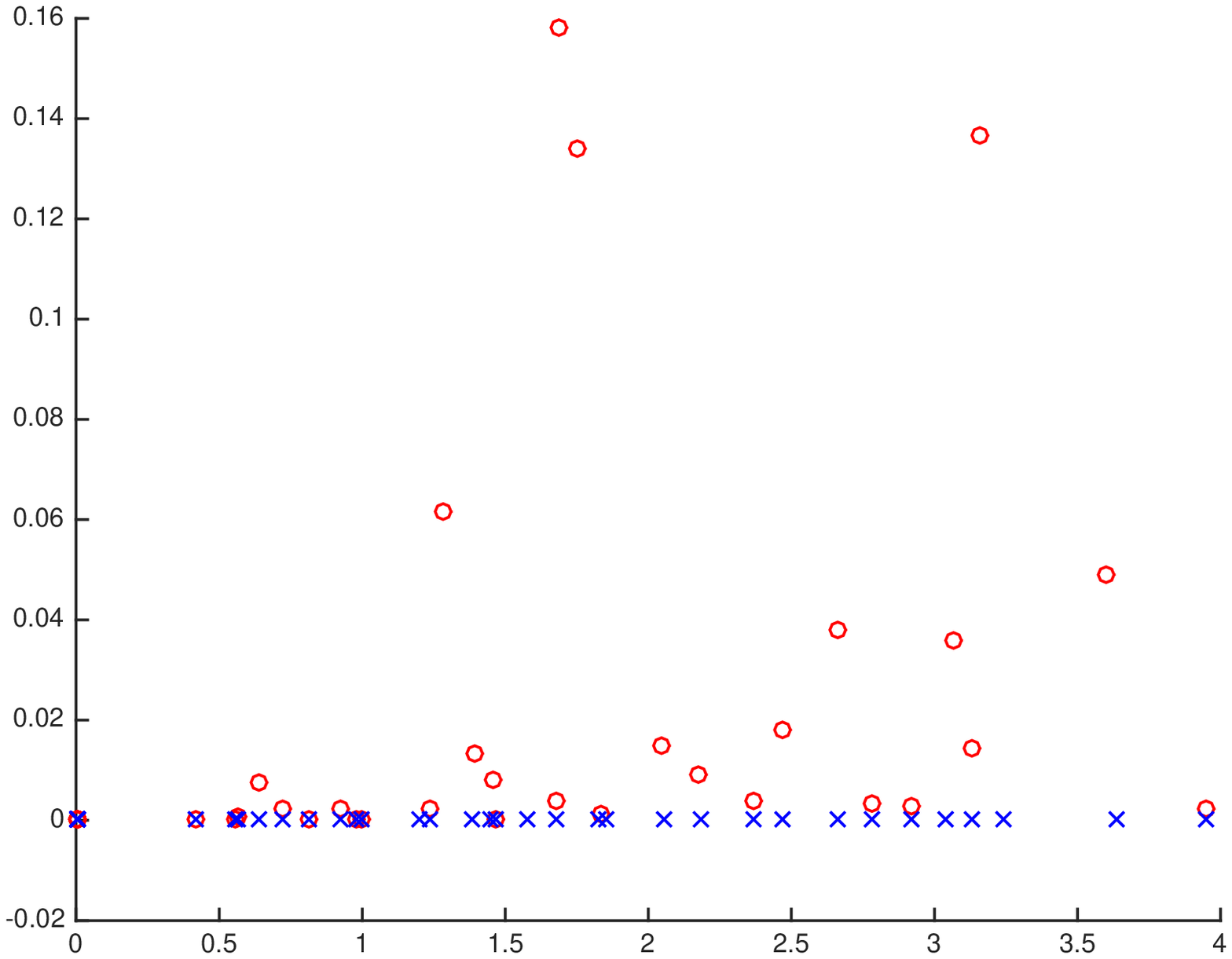}
\caption{{\bf All non-zero eigenvalues imaginary after a rank-1 perturbation of a random sample covariance matrix.}
Above is a plot of the eigenvalues (shown as $\times$ symbols) of $\mat M= \transp{\mat X} \mat X$, where $\mat X$ is the $30$ by $100$ random matrix with iid standard real Gaussian entries, and the eigenvalues (shown as $\circ$ symbols) of $\mat M(\mat I +\sqrt{-1}v v^*)$ where
$v=\transp{(1,0,0,\dots,0)}$.  There are 70 eigenvalues at the origin which are the same for both matrices, and the other $30$ eigenvalues for $\mat M(\mat I + \sqrt{-1}vv^*)$ have strictly positive imaginary parts, with minimum $3.3117 \times 10^{-8}$, illustrating Theorem~\ref{thm:sampcov-nonreal}.}
\label{fig:SampCov+ivvt}
\end{figure}


Below, we give a deterministic result with a similar flavor to Theorem~\ref{thm:nonreal} and Theorem~\ref{thm:sampcov-nonreal} which applies to any Hermitian matrix.  

\begin{theorem}\label{thm:nonreal-nonrandom}
Let $\mat M$ be an $n \times n$ Hermitian matrix with eigenvalues $\lambda_1,\lambda_2,
\dots,\lambda_n$ (including repetitions).  Assume that the $k$ eigenvalues $\lambda_1, \ldots,\allowbreak \lambda_k$ are distinct.  Then there exists a column vector $v$ such that $\mat M + \gamma \sqrt{-1} v {v}^*$ shares the $n-k$ eigenvalues $\lambda_{k+1},\dots ,\lambda_n$ with $\mat M$ and also the following holds:
\begin{itemize}
\item if $\gamma>0$, then $\mat M + \gamma \sqrt{-1} v {v}^* $ has $k$ eigenvalues with positive imaginary part, and
\item if $\gamma <0$, then $\mat M + \gamma \sqrt{-1} v{v}^* $  has $k$ eigenvalues with negative imaginary part.
\end{itemize}
Furthermore, there are many choices for the vector $v$: if $w_1,\dots,w_k$ are eigenvectors corresponding to the $k$ distinct eigenvalues $\lambda_1,\dots,\lambda_k$ of $\mat M$, then any $v=\sum_{j=1}^k z_j w_j$ suffices, so long as the complex numbers $z_j \ne 0$ for all $1\le j\le k$.
\end{theorem}

Theorem~\ref{thm:nonreal-nonrandom} has a natural corollary applying to multiplicative perturbations of the form $\mat M(\mat I + \mat P)$, where $\mat M$ can be any Hermitian matrix, including a sample covariance matrix (see Corollary~\ref{cor:multshift}).  In fact, we will prove Theorem~\ref{thm:sampcov-nonreal} essentially by combining a version of Theorem~\ref{thm:nonreal-nonrandom} 
(see Lemma~\ref{lem:nonreal-nonrandom}) with a lemma showing the necessary conditions on the eigenvalues and eigenvectors are satisfied with probability 1 (see Lemma~\ref{lemma:distinct}).

\subsection{Global behavior of the eigenvalues} \label{sec:global}
As Theorem \ref{thm:nonreal} shows, it is possible that no single eigenvalue of the sum $\mat{M} + \mat{P}$ is real, even when $\mat{M}$ is random and $\mat P$ has low rank.  However, we can still describe the limiting behavior of the eigenvalues.  We do so in Theorem~\ref{thm:global:sc} and Theorem~\ref{thm:global:mp} below, both of which are consequences of Theorem \ref{thm:kahan}.  

Recall that Wigner's semicircle law $\mu_{\mathrm{sc}}$ is the measure on $\mathbb{R}$ with density given in \eqref{eq:def:rhosc}.  Here and in the sequel, we view $\mu_{\mathrm{sc}}$ as a measure on $\mathbb{C}$.  In particular, Definition \ref{def:weakconvergence} defines what it means for a sequence of probability measures on $\mathbb{C}$ to converge to $\mu_{\mathrm{sc}}$.  We observe that, as a measure on $\mathbb{C}$, $\mu_{\mathrm{sc}}$ is not absolutely continuous (with respect to Lebesgue measure).  In particular, the density $\rho_{\mathrm{sc}}$ given in \eqref{eq:def:rhosc} is not the density of $\mu_{\mathrm{sc}}$ when viewed as a measure on $\mathbb{C}$.  However, if $f:\mathbb{C} \to \mathbb{C}$ is a bounded continuous function, then
$$ \int_{\mathbb{C}} f(z) d \mu_{\mathrm{sc}}(z) = \int_{\mathbb{R}} f(x) d \mu_{\mathrm{sc}}(x) = \int_{-2}^2 f(x) \rho_{\mathrm{sc}}(x) dx. $$

\begin{theorem}[Perturbed Wigner matrices] \label{thm:global:sc}
Let $\xi$ and $\zeta$ be real random variables, and assume $\xi$ has unit variance.  For each $n \geq 1$, let $\mat W_n$ be an $n \times n$ real symmetric Wigner matrix with atom variables $\xi$ and $\zeta$, and let $\mat P_n$ be an arbitrary $n \times n$ deterministic matrix.  If 
\begin{equation} \label{eq:Pnhs}
	\lim_{n \to \infty} \frac{1}{\sqrt{n}} \| \mat P_n \|_2 = 0, 
\end{equation}
then, with probability $1$, the empirical spectral measure $\mu_{\frac{1}{\sqrt{n}}\mat {W}_n + \mat P_n}$ of $\frac{1}{\sqrt{n}} \mat W_n + \mat P_n$ converges weakly on $\mathbb{C}$ as $n \to \infty$ to the (non-random) measure $\mu_{\mathrm{sc}}$.  
\end{theorem}

In the coming sections, we will typically consider the case when $\| \mat P_n\| = O(1)$ and $\rank(\mat P_n) = O(1)$.  In this case, it follows that $\|\mat P_n\|_2 = O(1)$, and hence \eqref{eq:Pnhs} is satisfied.  

We similarly consider the Marchenko--Pastur law $\mu_{\mathrm{MP},y}$ as a measure on $\mathbb{C}$.  For perturbed sample covariance matrices, we can also recover the Marchenko--Pastur law as the limiting distribution.  

\begin{theorem}[Perturbed sample covariance matrices] \label{thm:global:mp}
Let $\xi$ be a real random variable with mean zero, unit variance, and finite fourth moment.  For each $n \geq 1$, let $\mat S_n$ be a sample covariance matrix with atom variable $\xi$ and parameters $(m,n)$, where $m$ is a function of $n$ such that $y_n := \frac{m}{n} \to y \in (0, \infty)$ as $n \to \infty$.  Let $\mat P_n$ be an arbitrary $n \times n$ deterministic matrix.  If \eqref{eq:Pnhs} holds, then, with probability $1$, the empirical spectral measure $\mu_{\frac{1}{\sqrt{mn}} \mat S_n(\mat I + \mat P_n)}$ of $\frac{1}{\sqrt{mn}} \mat S_n (\mat I + \mat P_n)$ converges weakly on $\mathbb{C}$ as $n \to \infty$ to the (non-random) measure $\mu_{\mathrm{MP},y}$. 
\end{theorem}

\subsection{More refined behavior of the eigenvalues} \label{sec:refine}
While Theorems \ref{thm:global:sc} and \ref{thm:global:mp} show that 
all but a vanishing proportion
of the eigenvalues of $\mat{M} + \mat {P}$ converge to the real line, the results do not quantitatively answer Question~\ref{q:main}.  We now give a more quantitative bound on the imaginary part of the eigenvalues.  

We first consider the Wigner case.  For any $\delta > 0$, let $\Esc_{\delta}$ denote the $\delta$-neighborhood of the interval $[-2,2]$ in the complex plane.   That is, 
$$ \Esc_{\delta} := \left\{ z \in \mathbb{C} : \inf_{x \in [-2,2]} |z - x| \leq \delta \right\}. $$
Here, we work with the interval $[-2,2]$ as this is the support of the semicircle distribution $\mu_{\mathrm{sc}}$.  Our main results below are motivated by the following result from \cite{OR}, which describes the eigenvalues of $\mat M := \frac{1}{\sqrt{n}} \mat W + \mat P$ when $\mat W$ is a Wigner matrix and $\mat P$ is a deterministic matrix with rank $O(1)$.  In the numerical analysis literature \cite{K}, such perturbations of Hermitian matrices are sometimes referred to as \emph{nearly Hermitian matrices}.  

\begin{theorem}[Theorem 2.4 from \cite{OR}] \label{thm:or} 
Let $\xi$ and $\zeta$ be real random variables.  Assume $\xi$ has mean zero, unit variance, and finite fourth moment;  suppose $\zeta$ has mean zero and finite variance.  For each $n \geq 1$, let $\mat{W}_n$ be an $n \times n$ Wigner matrix with atom variables $\xi$ and $\zeta$.  Let $k \geq 1$ and $\delta > 0$.  For each $n \geq 1$, let $\mat P_n$ be an $n \times n$ deterministic matrix, where $\sup_{n \geq 1} \rank( \mat P_n) \leq k$ and $\sup_{n \geq 1} \| \mat P_n\| = O(1)$.  Suppose for $n$ sufficiently large, there are no nonzero eigenvalues of $\mat P_n$ which satisfy
$$ \lambda_i(\mat P_n) + \frac{1}{ \lambda_i(\mat P_n)} \in \Esc_{3\delta} \setminus \Esc_{\delta} \quad \text{with} \quad \abs{\lambda_i(\mat P_n)} > 1, $$
and there are $j$ eigenvalues $\lambda_1(\mat P_n), \ldots, \lambda_j(\mat P_n)$ for some $j \leq k$ which satisfy 
$$ \lambda_i(\mat P_n) + \frac{1}{ \lambda_i(\mat P_n)} \in \mathbb{C} \setminus \Esc_{3 \delta} \quad \text{with} \quad \abs{\lambda_i(\mat P_n)} > 1. $$
Then, almost surely, for $n$ sufficiently large, there are exactly $j$ eigenvalues of $\frac{1}{\sqrt{n}} \mat W_n + \mat P_n$ in the region $\mathbb{C} \setminus \Esc_{2\delta}$, and after labeling the eigenvalues properly, 
\begin{equation} \label{eq:outliereigdesc}
	\lambda_i \left( \frac{1}{\sqrt{n}} \mat W_n + \mat P_n \right) = \lambda_i(\mat P_n) + \frac{1}{\lambda_i(\mat P_n)} + o(1) 
\end{equation}
for each $1 \leq i \leq j$.  
\end{theorem}

See Figure~\ref{fig:WigOut} for a numerical example illustrating Theorem~\ref{thm:or}.  Recently, Rochet \cite{R} has obtained the rate of convergence for the $j$ eigenvalues characterized in \eqref{eq:outliereigdesc} as well as a description of their fluctuations.

\begin{figure}
\includegraphics[width=\textwidth]{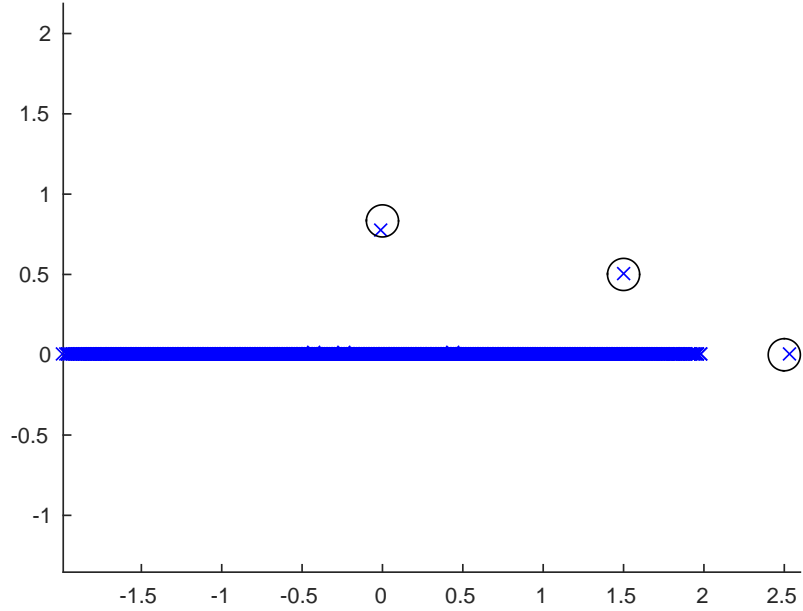}
\caption{{\bf Wigner outliers and Theorem~\ref{thm:or}.}
Above is a plot of the eigenvalues (marked by x's) of a 2000 by 2000 matrix $\frac1{\sqrt{n}}\mat W + \mat P$, where $\mat W$ is a GOE matrix
 and $\mat P = 
\diag(\frac32i,1+i,2,0,0,\dots, 0)$.  
Three circles of radius $0.1$ have been drawn centered at $(0,\frac56)$, 
$(\frac32, 1/2)$, and $(\frac52,0)$, which are the locations of the
outliers predicted by Theorem~\ref{thm:or}.
}
\label{fig:WigOut}
\end{figure}

While Theorem \ref{thm:or} describes the location of the ``outlier'' eigenvalues, it says nothing substantial about the $n-j$ eigenvalues in $\Esc_{2\delta}$.  We address this point in the following theorem.  We will require that the atom variables $\xi$ and $\zeta$ satisfy the following assumptions.  

\begin{definition}[Condition {\bf C0}] \label{def:C0}
We say the atom variables $\xi$ and $\zeta$ satisfy condition {\bf C0} if either one of the following conditions hold.  
\begin{enumerate}[(i)]
\item $\xi$ and $\zeta$ both have mean zero, unit variance, and finite moments of all orders; that is, for every non-negative integer $p$, there exists a constant $C_p$ such that
$$ \E|\xi|^p + \E|\zeta|^p \leq C_p. $$ 
\item $\xi$ and $\zeta$ both have mean zero, $\xi$ has unit variance, $\zeta$ has variance $2$, $\E[\xi^3] = \E[\zeta^3] = 0$, and both $\xi$ and $\zeta$ have sub-exponential tails, that is, there exists $\vartheta > 0$ such that
$$ \Prob (|\xi| \geq t) \leq \vartheta^{-1} \exp(-t^\vartheta) \text{ and } \Prob(|\zeta| \geq t) \leq \vartheta^{-1} \exp(-t^{\vartheta}) $$
for all $t > 0$.
\end{enumerate}
\end{definition}

\begin{theorem}[Location of the eigenvalues for Wigner matrices] \label{thm:wigner:refine}
Let $\xi$ and $\zeta$ be real random variables which satisfy condition {\bf C0}.  For each $n \geq 1$, let $\mat{W}_n$ be an $n \times n$ Wigner matrix with atom variables $\xi$ and $\zeta$.  Let $k \geq 1$ and $\delta > 0$.  For each $n \geq 1$, let $\mat P_n$ be an $n \times n$ deterministic matrix, where $\sup_{n \geq 1} \rank( \mat P_n) \leq k$ and $\sup_{n \geq 1} \| \mat P_n\| = O(1)$.  Suppose for $n$ sufficiently large, there are no eigenvalues of $\mat P_n$ which satisfy 
$$ |1 - |\lambda_i(\mat P_n)|| < \delta $$
and there are $j$ eigenvalues $\lambda_1(\mat P_n), \ldots, \lambda_j(\mat P_n)$ for some $j \leq k$ which satisfy 
\begin{equation} \label{eq:largeevalues}
	|\lambda_i(\mat P_n)| \geq 1 + \delta. 
\end{equation}
Then, there exists a constant $\delta'$ satisfying $0< \delta' < \frac{\delta^2}{1+\delta}$ such that, for every $\eps > 0$, the following holds.  Almost surely, for $n$ sufficiently large, there are exactly $j$ eigenvalues of $\frac{1}{\sqrt{n}} \mat W_n + \mat P_n$ in the region $\mathbb{C} \setminus \Esc_{\delta'}$, and after labeling the eigenvalues properly, 
$$ \lambda_i \left( \frac{1}{\sqrt{n}} \mat W_n + \mat P_n \right) = \lambda_i(\mat P_n) + \frac{1}{\lambda_i(\mat P_n)} + o(1) $$
for each $1 \leq i \leq j$.  In addition, almost surely, for $n$ sufficiently large, the remaining eigenvalues of $\frac{1}{\sqrt{n}} \mat W_n + \mat P_n$ satisfy
$$ \sup_{j+1 \leq i \leq n} \left|\Im \left( \lambda_i \left( \frac{1}{\sqrt{n}} \mat W_n + \mat P_n \right) \right) \right| \leq n^{-1 + \eps} $$
and
$$ \sup_{j+1 \leq i \leq n} \left| \Re \left( \lambda_i \left( \frac{1}{\sqrt{n}} \mat W_n + \mat P_n \right) \right) \right| \leq 2 + n^{-2/3 + \eps}. $$
\end{theorem}

In other words, Theorem \ref{thm:wigner:refine} states that besides for the ``outlier'' eigenvalues noted in Theorem \ref{thm:or}, all of the other eigenvalues are within $n^{-1}$ of the real line, up to $n^{\eps}$ multiplicative corrections.  In addition, the real parts of the eigenvalues highly concentrate in the region $[-2,2]$.  

Similar results to Theorem \ref{thm:wigner:refine} have also appeared in the mathematical physics literature due to the role random matrix theory plays in describing scattering in chaotic systems.  We refer the interested reader to \cite{FSphys1996,FSphys1997,FKphys1999,FSphys2003,SFTphys1999} and references therein.  In particular, these works focus on the case when $\mat W_n$ is drawn from the GOE or it complex Hermitian relative, the Gaussian unitary ensemble (GUE), and $\mat P_n$ is of the form $\mat P_n = -\sqrt{-1} \mat \Gamma_n$, where $\mat\Gamma_n$ is a positive definite matrix of low rank.  We emphasis that the methods used in \cite{FSphys1996,FSphys1997,FKphys1999,FSphys2003,SFTphys1999} only apply to the GUE and GOE, while Theorem \ref{thm:wigner:refine} applies to a large class of Wigner matrices.  

We now consider the sample covariance case.  For simplicity, we will consider sample covariance matrices with parameters $(n,n)$, that is, sample covariance matrices $\mat S_n$ of the form $\mat S_n = \mat X^\mathrm{T} \mat X$, where $\mat X$ is $n \times n$.  For any $\delta > 0$, let $\Emp_{\delta}$ be the $\delta$-neighborhood of $[0,4]$ in the complex plane.  That is, 
$$ \Emp_{\delta} := \left\{ z \in \mathbb{C} : \inf_{x \in [0,4]} |z - x| \leq \delta \right\}. $$
Here, we work on $[0,4]$ as this is the support of $\mu_{\mathrm{MP},1}$.  Our main result for sample covariance matrices is the following theorem.  We assume the atom variable $\xi$ of $\mat S_n$ satisfies the following condition.

\begin{definition}[Condition {\bf C1}] \label{def:C1}
We say the atom variable $\xi$ satisfies condition {\bf C1} if $\xi$ has mean zero, unit variance, and finite moments of all orders, that is, for every non-negative integer $p$ there exists a constant $C_p > 0$ such that
$$ \E |\xi|^p \leq C_p. $$
\end{definition}

\begin{theorem}[Location of the eigenvalues for sample covariance matrices] \label{thm:sc:refine}
Let $\xi$ be a real random variable which satisfies condition {\bf C1}.  For each $n \geq 1$, let $\mat{S}_n$ be an $n \times n$ sample covariance matrix with atom variable $\xi$ and parameters $(n,n)$.  Let $k \geq 1$ and $\delta > 0$.  For each $n \geq 1$, let $\mat P_n$ be an $n \times n$ deterministic matrix, where $\sup_{n \geq 1} \rank( \mat P_n) \leq k$ and $\sup_{n \geq 1} \| \mat P_n\| = O(1)$.  Suppose for $n$ sufficiently large, there are no eigenvalues of $\mat P_n$ which satisfy 
\begin{equation} \label{eq:sc:noeigdelta}
	|1 - |\lambda_i(\mat P_n)|| < \delta 
\end{equation}
and there are $j$ eigenvalues $\lambda_1(\mat P_n), \ldots, \lambda_j(\mat P_n)$ for some $j \leq k$ which satisfy 
\begin{equation} \label{eq:sc:largeevalues}
	|\lambda_i(\mat P_n)| \geq 1 + \delta. 
\end{equation}
Then, there exists a constant $\delta'$ satisfying $0< \delta'<\frac{\delta^2}{1+\delta}$ such that, for every $\eps > 0$, the following holds.  Almost surely, for $n$ sufficiently large, there are exactly $j$ eigenvalues of $\frac{1}{n} \mat S_n (\mat I + \mat P_n)$ in the region $\mathbb{C} \setminus \Emp_{\delta'}$, and after labeling the eigenvalues properly, 
$$ \lambda_i \left( \frac{1}{n} \mat S_n (\mat I+ \mat P_n) \right) = \lambda_i(\mat P_n) \left(1 + \frac{1}{\lambda_i(\mat P_n)} \right)^2 + o(1) $$
for each $1 \leq i \leq j$.  In addition, almost surely, for $n$ sufficiently large, the remaining eigenvalues of $\frac{1}{n} \mat S_n (\mat I+ \mat P_n)$ either lie in the disk $\{z \in \mathbb{C} : |z| \leq \delta' \}$ or the eigenvalues satisfy
$$ \sup_{j+1 \leq i \leq n} \left|\Im \left( \lambda_i \left( \frac{1}{{n}} \mat S_n(\mat I + \mat P_n) \right) \right) \right| \leq n^{-1 + \eps}, $$
$$ \inf_{j+1 \leq i \leq n} \Re \left( \lambda_i \left( \frac{1}{n} \mat S_n(\mat I + \mat P_n) \right) \right) > 0, $$
and
$$ \sup_{j+1 \leq i \leq n} \Re \left( \lambda_i \left( \frac{1}{{n}} \mat S_n(\mat I + \mat P_n) \right) \right) \leq 4 + n^{-2/3 + \eps}. $$
\end{theorem}

See Figure~\ref{fig:SampCovOut} for a numerical demonstration of Theorem~\ref{thm:sc:refine}.

\begin{figure}
\includegraphics[width=\textwidth]{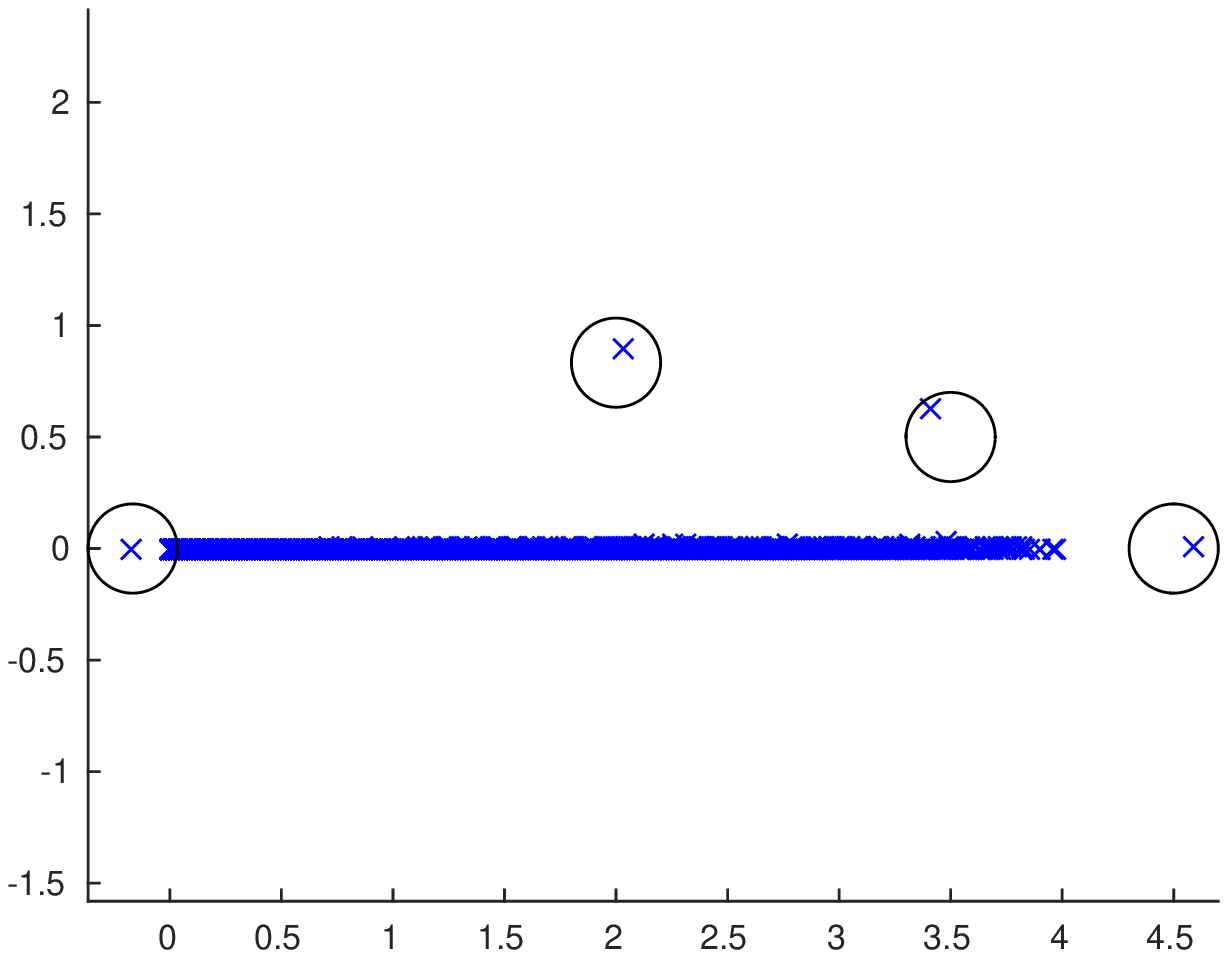}
\caption{{\bf Sample covariance outliers and Theorem~\ref{thm:sc:refine}.}
Above is a plot of the eigenvalues (marked by x's) of a 2000 by 2000 matrix $\frac1{n}\transp{\mat X}\mat X(\mat I + \mat P)$, where $\mat X$ has iid standard real Gaussian entries and $\mat P = 
\diag(-\frac32,\frac32i,1+i,2,0,0,\dots, 0)$.  
Three circles of radius $0.2$ have been drawn centered at $(-\frac16,0)$, $(2,\frac56)$, 
$(\frac72, \frac12)$, and $(\frac92,0)$, which are the locations of the
outliers predicted by Theorem~\ref{thm:sc:refine}.
}
\label{fig:SampCovOut}
\end{figure}

\subsection{Eigenvectors of random perturbations} \label{sec:eigenvectors}
In this section, we mention a few results concerning the eigenvectors of perturbed Wigner matrices.  In particular, using the methods developed by Benaych-Georges and Nadakuditi \cite{BGN}, when the perturbation $\mat P_n$ is random (and independent of $\mat W_n$), we can say more about the eigenvectors of $\frac{1}{\sqrt{n}}\mat W_n + \mat P_n$.  For simplicity, we consider the case when $\mat P_n$ is a rank one normal matrix.  

\begin{theorem}[Eigenvectors: Wigner case] \label{thm:eigenvectors}
Let $\xi$ and $\zeta$ be real random variables.  Assume $\xi$ has mean zero, unit variance, and finite fourth moment;  suppose $\zeta$ has mean zero and finite variance.  For each $n \geq 1$, let $\mat{W}_n$ be an $n \times n$ Wigner matrix with atom variables $\xi$ and $\zeta$.  In addition, for each $n \geq 1$, let $u_n$ be a random vector uniformly distributed on the unit sphere in $\mathbb{R}^n$ or, respectively, in $\mathbb{C}^n$; and let $u_n$ be independent of $\mat W_n$.  Set $\mat{P}_n = \theta u_n u_n^\ast$, where $\theta \in \mathbb C$ and $|\theta| > 1$.  Then there exists $\delta > 0$ (depending on $\theta$) such that the following holds.  Almost surely, for $n$ sufficiently large, there is exactly one eigenvalue of $\frac{1}{\sqrt{n}} \mat W_n + \mat P_n$ outside $\Esc_{\delta}$, and this eigenvalue takes the value $\theta + \frac{1}{\theta} + o(1)$.  Let $v_n$ be the unit eigenvector corresponding to this eigenvalue.  Then, almost surely, 
$$ |u_n^\ast v_n|^2 \longrightarrow \frac{ \left| \int \frac{\rho_{\mathrm{sc}}(x)}{x - \tilde{\theta}}\,  dx \right|^2 }{ \int \frac{\rho_{\mathrm{sc}}(x)}{|x - \tilde{\theta}|^2}\, dx} $$
as $n \to \infty$, where $\tilde{\theta} := \theta + \frac{1}{\theta}$ and $\rho_{\mathrm{sc}}$ is the density of the semicircle law given in \eqref{eq:def:rhosc}.  
\end{theorem}

The first part of Theorem \ref{thm:eigenvectors} follows from Theorem \ref{thm:or}.  The second part, regarding the unit eigenvector $v_n$ is new.  In particular, Theorem \ref{thm:eigenvectors} describes the portion of the vector $v_n$ pointing in direction $u_n$.  In the case when $\mat P_n$ is Hermitian (i.e. $\theta \in \mathbb{R}$), we recover a special case of \cite[Theorem 2.2]{BGN}.  We also have the following version for the sample covariance case.  

\begin{theorem}[Eigenvectors: sample covariance case] \label{thm:SC:eigenvectors}
Let $\xi$ be a real random variable which satisfies condition {\bf C1}.  For each $n \geq 1$, let $\mat{S}_n$ be an $n \times n$ sample covariance matrix with atom variables $\xi$ and parameters $(n,n)$.  In addition, for each $n \geq 1$, let $u_n$ be a random vector uniformly distributed on the unit sphere in $\mathbb{R}^n$ or, respectively, in $\mathbb{C}^n$; and let $u_n$ be independent of $\mat S_n$.  Set $\mat{P}_n = \theta u_n u_n^\ast$, where $\theta \in \mathbb C$ and $|\theta| > 1$.  Then there exists $\delta > 0$ (depending on $\theta$) such that the following holds.  Almost surely, for $n$ sufficiently large, there is exactly one eigenvalue of $\frac{1}{n} \mat{S}_n(\mat I + \mat P_n)$ outside $\Emp_{\delta}$, and this eigenvalue takes the value $\theta \left( 1 + \frac{1}{\theta} \right)^2 + o(1)$.  Let $v_n$ be the unit eigenvector corresponding to this eigenvalue.  Then, almost surely, 
$$ |u_n^\ast v_n|^2 \longrightarrow \frac{ \left| \int \frac{x\rho_{\mathrm{MP},1}(x)}{x - \hat{\theta}}\,  dx \right|^2 }{ \int \frac{x^2\rho_{\mathrm{MP},1}(x)}{|x - \hat{\theta}|^2}\, dx} $$
as $n \to \infty$, where $\hat{\theta} := \theta\left(1 + \frac{1}{\theta}\right)^2$ and $\rho_{\mathrm{MP},1}$ is the density of the Marchenko-Pastur law given in \eqref{eq:def:rhoMP} and \eqref{eq:def:lambdapm} with $y = 1$.  
\end{theorem}

\begin{remark}
Both Theorem \ref{thm:eigenvectors} and Theorem \ref{thm:SC:eigenvectors} can be extended to the case when $\mat P_n$ is non-normal.  For instance, Theorem \ref{thm:eigenvectors} holds when $\mat P_n := \sigma u_n w_n^\ast$, where $\sigma > 0$, where $u_n$ and $w_n$ are unit vectors, where $u_n$ is a random vector uniformly distributed on the unit sphere in $\mathbb{R}^n$ or, respectively, in $\mathbb{C}^n$, and where  $(u_n, w_n)$ is independent of $\mat W_n$ (but $u_n$ and $w_n$ are not necessarily assumed independent of each other).  In this case, $\mat P_n$ has eigenvalue $\theta := \sigma w_n^\ast u_n$ with corresponding unit eigenvector $u_n$.  Since $\theta$ is now random and (possibly) dependent on $n$, it must be additionally assumed that almost surely $| \theta | > 1 + \eps$ for some $\eps > 0$.  Similarly, both theorems can also be extended to the case when $\mat P_n$ has rank larger than one.  These extensions can be proved by modifying the arguments presented in Section \ref{sec:proof:eigenvectors}.
\end{remark}

\subsection{Critical points of characteristic polynomials} \label{sec:critical}
As an application of our main results, we study the critical points of characteristic polynomials of random matrices.  Recall that a \emph{critical point} of a polynomial $f$ is a root of its derivative $f'$.  There are many results concerning the location of critical points of polynomials whose roots are known.  For example, the famous Gauss--Lucas theorem offers a geometric connection between the roots of a polynomial and the roots of its derivative.  

\begin{theorem}[Gauss--Lucas; Theorem 6.1 from \cite{M}] \label{thm:gauss}
If $f$ is a non-constant polynomial with complex coefficients, then all zeros of $f'$ belong to the convex hull of the set of zeros of $f$.  
\end{theorem}

Pemantle and Rivin \cite{PR} initiated the study of a probabilistic version of the Gauss--Lucas theorem.  In particular, they studied the critical points of the polynomial 
\begin{equation} \label{eq:pnprod}
	p_n(z) := (z-X_1) \cdots (z-X_n)
\end{equation}
when $X_1, \ldots, X_n$ are iid complex-valued random variables.  Their results were later generalized by Kabluchko \cite{Kcrit} to the following.  

\begin{theorem}[Theorem 1.1 from \cite{Kcrit}] \label{thm:PMK}
Let $\mu$ be any probability measure on $\mathbb{C}$.  Let $X_1, X_2, \ldots$ be a sequence of iid random variables with distribution $\mu$.  For each $n \geq 1$, let $p_n$ be the degree $n$ polynomial given in \eqref{eq:pnprod}.  Then the empirical measure constructed from the critical points of $p_n$ converges weakly to $\mu$ in probability as $n \to \infty$.  
\end{theorem}

Theorem \ref{thm:PMK} was later extended by the first author; indeed, in \cite{O}, the critical points of characteristic polynomials for certain classes of random matrices drawn from the compact classical matrix groups are studied.  Here, we extend these results to classes of perturbed random matrices.  For an $n \times n$ matrix $\mat{A}$, we let $p_{\mat{A}}(z) := \det(\mat{A} - z \mat{I})$ denote the characteristic polynomial of $\mat{A}$.  In this way, the empirical spectral measure $\mu_{\mat{A}}$ can be viewed as the empirical measure constructed from the roots of $p_{\mat A}$.  Let $\mu_{\mat{A}}'$ denote the empirical measure constructed from the critical points of $p_{\mat A}$.  That is,
$$ \mu_{\mat{A}}' := \frac{1}{n-1} \sum_{k=1}^{n-1} \delta_{\xi_k}, $$
where $\xi_1, \ldots, \xi_{n-1}$ are the critical points of $p_{\mat{A}}$ counted with multiplicity.  

We prove the following result for Wigner random matrices.

\begin{theorem} \label{thm:wigner:critical}
Assume $\xi$, $\zeta$, $\mat W_n$, and $\mat P_n$ satisfy the assumptions of Theorem \ref{thm:wigner:refine}.  In addition, for $n$ sufficiently large, assume that the $j$ eigenvalues of $\mat P_n$ which satisfy \eqref{eq:largeevalues} do not depend on $n$.  Then, almost surely, $\mu_{\frac{1}{\sqrt{n}} \mat W_n + \mat P_n}'$ converges weakly on $\mathbb{C}$ as $n \to \infty$ to the (non-random) measure $\mu_{\mathrm{sc}}$.
\end{theorem}

A numerical example of Theorem~\ref{thm:wigner:critical} appears in Figure~\ref{fig:wig:crit}.

\begin{remark}
Due to the outlier eigenvalues (namely, those eigenvalues in $\mathbb C \setminus \Esc_{\delta'}$) observed in Theorems \ref{thm:or} and \ref{thm:wigner:refine}, the convex hull of the eigenvalues of $\frac{1}{\sqrt{n}} \mat W_n + \mat{P}_n$ can be quite large.  In particular, it does not follow from the Gauss--Lucas theorem (Theorem \ref{thm:gauss}) that the majority of the critical points converge to the real line.  On the other hand, besides describing the limiting distribution, Theorem \ref{thm:wigner:critical} asserts that all but $o(n)$ of the critical points converge to the real line.  
\end{remark}

\begin{figure}
\includegraphics[width=\textwidth]{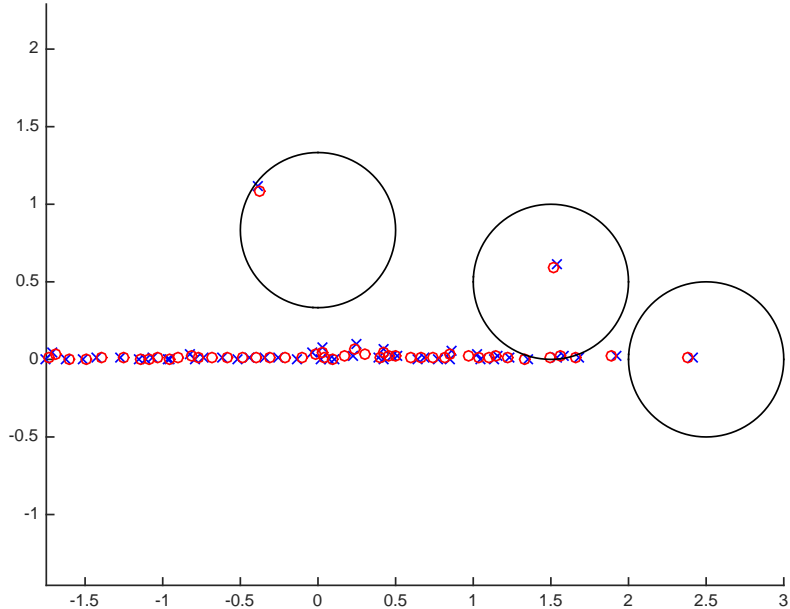}
\caption{ {\bf Critical points and eigenvalues.}
Above is a plot of the eigenvalues (marked by x's) of a 50 by 50 matrix $\frac1{\sqrt{n}}\mat W + \mat P$, where $\mat W$ is a Wigner matrix with standard real Gaussian atom variables and $\mat P = 
\diag(\frac32i,1+i,2,0,0,\dots, 0)$.  The critical points (marked by o's) of the characteristic polynomial of the matrix $\frac1{\sqrt{n}}\mat W + \mat P$ are also plotted above.  Note that the critical points match closely with the eigenvalues, demonstrating Theorem~\ref{thm:wigner:critical}.
Three circles of radius $0.5$ have been drawn centered at $(0,\frac56)$, 
$(\frac32, 1/2)$, and $(\frac52,0)$, which are the locations of the
outliers predicted by Theorem~\ref{thm:or}; notably, there appear to be outlier critical points that match closely with the outlier eigenvalues.
}
\label{fig:wig:crit}
\end{figure}

\subsection{Outline}

The rest of the paper is devoted to the proofs of our main results.  Section \ref{sec:proof:nonreal} is devoted to the proofs of Theorem~\ref{thm:nonreal}, Theorem~\ref{thm:sampcov-nonreal}, and Theorem~\ref{thm:nonreal-nonrandom}.  The results in Section \ref{sec:global} are proven in Section \ref{sec:proof:global}.  Section \ref{sec:proof:refine} contains the proof of our main result from Section \ref{sec:refine}.  Theorem \ref{thm:eigenvectors} and Theorem~\ref{thm:SC:eigenvectors} are proven in Section \ref{sec:proof:eigenvectors}.  We prove Theorem \ref{thm:wigner:critical} in Section \ref{sec:proof:critical}.  Lastly, the appendix contains a number of auxiliary results.

\section{Proof of Theorem \ref{thm:nonreal}, Theorem \ref{thm:sampcov-nonreal}, and Theorem~\ref{thm:nonreal-nonrandom}} \label{sec:proof:nonreal}

\subsection{Proof of Theorem \ref{thm:nonreal}}
We begin this section with a proof of Theorem \ref{thm:nonreal}.  The proof of Theorem~\ref{thm:nonreal-nonrandom} is in Subsection~\ref{ss:nonreal-nonrandom}, and the proof of Theorem~\ref{thm:sampcov-nonreal} is in Subsection~\ref{ss:sampcov-nonreal}.   

Recall that two non-constant polynomials $P$ and $Q$ with real coefficients have \emph{weakly interlacing} zeros if
\begin{itemize}
\item their degrees are equal or differ by one,
\item their zeros are all real, and
\item there exists an ordering such that
\begin{equation} \label{eq:def:interlace}
	\alpha_1 \leq \beta_1 \leq \alpha_2 \leq \beta_2 \leq \cdots \leq \alpha_s \leq \beta_s \leq \cdots, 
\end{equation}
where $\alpha_1, \alpha_2, \ldots$ are the zeros of one polynomial and $\beta_1, \beta_2, \ldots$ are those of the other.  
\end{itemize}
If, in the ordering \eqref{eq:def:interlace}, no equality sign occurs, then $P$ and $Q$ have \emph{strictly interlacing} zeros.  
Analogously, we say two Hermitian matrices have weakly or strictly interlacing eigenvalues if the respective interlacing property holds for the zeros of their characteristic polynomials.  

We begin with the following criterion, which completely characterizes when the zeros of two polynomials have strictly interlacing zeros.  
\begin{theorem}[Hermite--Biehler; Theorem 6.3.4 from \cite{RS}] \label{thm:hb}
Let $P$ and $Q$ be non-constant polynomials with real coefficients.  Then $P$   and $Q$ have strictly interlacing zeros if and only if the polynomials $f(z) := P(z) + \sqrt{-1} Q(z)$ has all its zeros in the upper half-plane $\mathbb{C}^+$ or in the lower half-plane $\mathbb{C}^{-}$.
\end{theorem}

We now extend Theorem \ref{thm:hb} to the characteristic polynomial of a perturbed Hermitian matrix.  

\begin{lemma} \label{lemma:hbiff}
Let $n \geq 2$.  Suppose $\mat{A}$ is an $n \times n$ Hermitian matrix of the form
$$ \mat{A} = \begin{bmatrix} \mat{B} & X \\ X^\ast & d \end{bmatrix}, $$
where $\mat B$ is a $(n-1) \times (n-1)$ Hermitian matrix, $X \in \mathbb{C}^{n-1}$, and $d \in \mathbb{R}$.  Let $\mat P$ be the diagonal matrix $\mat P :=    \diag(0, \ldots, 0, \gamma \sqrt{-1})$ for some $\gamma \in \mathbb{R}$ with $\gamma \neq 0$.  Then all eigenvalues of $\mat{A} + \mat{P}$ are in the upper half-plane $\mathbb{C}^+$ or the lower half-plane $\mathbb{C}^{-}$ if and only if the eigenvalues of $\mat{A}$ and $\mat{B}$ strictly interlace.
\end{lemma}
\begin{proof}
We observe that the characteristic polynomial of $\mat{A} + \mat{P}$ can be written as
\begin{align*}
    \det(\mat{A} + \mat{P} - z \mat{I}) &= \det \begin{pmatrix} \mat{B} - z     \mat{I} & X \\ X^\ast & d - z + \sqrt{-1} \gamma \end{pmatrix} \\
    &= \det \begin{pmatrix} \mat{B} - z \mat{I} & X \\ X^\ast & d-z             \end{pmatrix} + \det \begin{pmatrix} \mat{B} - z \mat{I} & X \\ 0 & \sqrt{-1} \gamma          \end{pmatrix} \\
    &= \det(\mat{A}- z \mat{I}) + \sqrt{-1} \gamma \det(\mat{B} - z \mat{I})
\end{align*}
by linearity of the determinant.  Since $\mat{A}$ and $\mat{B}$ are Hermitian   matrices, it follows that their characteristic polynomials $\det(\mat{A}- z     \mat{I})$ and $\det(\mat{B} - z \mat{I})$ are non-constant polynomials with real coefficients.  Hence, the polynomial $\gamma \det(\mat{B} - z \mat{I})$ is also a non-constant polynomial with real coefficients.  Thus, the claim follows from Theorem \ref{thm:hb}.
\end{proof}

In order to prove Theorem \ref{thm:nonreal}, we will also need the following lemma, which is based on the arguments of Tao-Vu \cite{TVsimple} and Nguyen-Tao-Vu \cite{NTV}.

\begin{lemma}[Strict eigenvalue interlacing] \label{lemma:interlace}
Let $\mu > 0$, and let $\xi$ be a real random variables such that \eqref{eq:nondeg} holds.  Let $\zeta$ be an arbitrary real random variable.  Suppose $\mat W$ is an $n \times n$ Wigner matrix with atom variables $\xi$ and $\zeta$.  Let $\mat{W}'$ be the upper-left $(n-1) \times (n-1)$ minor of $\mat W$.  Then, for any $\alpha > 0$, there exists $C > 0$ (depending on $\alpha$ and $\mu$) such that, with probability at least $1 - C n^{-\alpha}$, the eigenvalues of $\mat{W}$ strictly interlace with the eigenvalues of $\mat{W}'$.  

Furthermore, if $\xi$ and $\zeta$ are absolutely continuous random variables, then the eigenvalues of $\mat{W}$ and $\mat W'$ strictly interlace with probability $1$.  
\end{lemma}

We present a proof of Lemma \ref{lemma:interlace} in Appendix \ref{app:interlace} based on the arguments from \cite{NTV,TVsimple}.  We now complete the proof of Theorem \ref{thm:nonreal}.

\begin{proof}[Proof of Theorem \ref{thm:nonreal}]
Let $\mat{W}'$ be the upper-left $(n-1) \times (n-1)$ minor of $\mat{W}$.  We decompose $\mat{W} + \sqrt{n} \mat P$ as 
$$ \mat{W} + \sqrt{n} \mat P = \begin{bmatrix} \mat W' & X \\ X^\mathrm{T} & w_{nn} + \sqrt{-1} \gamma \sqrt{n} \end{bmatrix}, $$
where $X$ is the $n$-th column of $\mat W$ with the $n$-th entry removed and $w_{nn}$ is the $(n,n)$-entry of $\mat W$.

Let $\alpha > 0$, and let $\JOmega$ be the event where the eigenvalues of $\mat W$ and $\mat W'$ strictly interlace.  Lemma~\ref{lemma:interlace} implies that there exists $C > 0$ (depending on $\mu$ and $\alpha$) such that 
$$ \Prob( \JOmega ) \geq 1 - C n^{-\alpha}. $$
When $\xi$ and $\zeta$ are absolutely continuous random variables, the second part of Lemma \ref{lemma:interlace} implies that $\JOmega$ holds with probability $1$.  Moreover, on the event $\JOmega$, Lemma \ref{lemma:hbiff} implies that all the eigenvalues of $\mat W + \sqrt{n} \mat P$ are in the upper half-plane or the lower half-plane.  Hence, the same is true for the eigenvalues of $\frac{1}{\sqrt{n}} \mat W + \mat P$ on the event $\JOmega$.  

It remains to show that the choice of $\mathbb{C}^+$ or $\mathbb{C}^-$ is determined by the sign of $\gamma$.  We observe that
$$ \Im \left( \tr \left( \frac{1}{\sqrt{n}} \mat W + \mat P \right) \right) = \gamma $$
since $\mat{W}$ is a real-symmetric matrix.  As the trace of a matrix is the sum of its eigenvalues, it follows that, on the event $\JOmega$, 
\begin{itemize}
\item if $\gamma > 0$, then all the eigenvalues of $\frac{1}{\sqrt{n}} \mat W + \mat P$ are in $\mathbb{C}^{+}$,
\item if $\gamma < 0$, then all the eigenvalues of $\frac{1}{\sqrt{n}} \mat W + \mat P$ are in $\mathbb{C}^{-}$.
\end{itemize}
The proof of the theorem is complete. 
\end{proof}

\subsection{Proof of Theorem~\ref{thm:nonreal-nonrandom}} \label{ss:nonreal-nonrandom}

We will prove Theorem~\ref{thm:nonreal-nonrandom} from the slightly stronger result below. 

\begin{theorem}\label{thm:nonreal-nonrandom-diag}
Let $\mat M$ be an $n \times n$ Hermitian matrix with eigenvalues $\lambda_1,\lambda_2,
\dots,\lambda_n$ (including repetitions).  Assume that the $k$ eigenvalues $\lambda_1,
\dots,\lambda_k$ are distinct, with corresponding unit eigenvectors $w_1, \dots, w_k$.  
Let $(z_1,\dots, z_k)$  be a list of non-zero complex scalers, let  $(a_1, \dots, a_k)$ be a list of non-zero real scalers, and define the vectors $v:=\sum_{i=1}^k z_i w_i$ and $u:=\sum_{i=1}^k a_i z_i w_i$.  Then the matrix $\mat M + \sqrt{-1} u {v}^*$ shares the $n-k$ eigenvalues $\lambda_{k+1},\dots ,\lambda_n$ with $\mat M$ and also the following holds:
\begin{itemize}
\item if $a_i > 0$ for $i=1,\dots, k$, then $\mat M + \sqrt{-1} u {v}^* $ has exactly $k$ eigenvalues with positive imaginary part, and
\item if $a_i < 0$ for $i=1,\dots, k$, then $\mat M + \sqrt{-1} u {v}^* $  has exactly $k$ eigenvalues with negative imaginary part. 
\end{itemize}
\end{theorem}
Note that Theorem~\ref{thm:nonreal-nonrandom} follows from Theorem~\ref{thm:nonreal-nonrandom-diag} by taking $a_i= \gamma$ for $i=1,\dots,k$.

We will prove Theorem~\ref{thm:nonreal-nonrandom-diag} by way of the following lemma.

\begin{lemma}\label{lem:nonreal-nonrandom}
Let $\mat D$ be a real diagonal $n \times n$ matrix with the first $k$ diagonal entries distinct real numbers and with no constraints on the remaining diagonal entries, let $\mat E$ be a real diagonal $n \times n$ matrix with the first $k$ diagonal entries strictly positive and all other diagonal entries equal to zero, and let $u$ be any column vector with the first $k$ entries non-zero complex numbers and no constraints on the remaining entries.  Then for any real number $\gamma$ the following holds:
\begin{itemize}
\item if $\gamma >0$, then $\mat D + \gamma\mat E \sqrt{-1} u {u}^* $ has exactly $k$ eigenvalues with positive imaginary part, 
\item if $\gamma<0$, then $\mat D + \gamma \mat E \sqrt{-1} u{u}^* $  has exactly $k$ eigenvalues with negative imaginary part, and
\end{itemize}
regardless of the sign of $\gamma$, the matrix $\mat D + \gamma \mat E \sqrt{-1} u{u}^* $  has 
$n-k$ eigenvalues matching the last $n-k$ diagonal entries in $\mat D$.
\end{lemma}

\newcommand{\Span}{\operatorname{Span}}
\newcommand{\partinv}[1]{#1^{[-1]}}
\newcommand{\partI}{\mat I^{[k]}}

\begin{proof}
Let $\partinv{\mat E}$ denote the partial inverse of $\mat E$ where non-zero entries are inverted and zero entries remain the same.  Thus $\partinv{\mat E} \mat E= \partI$, the diagonal matrix with the first $k$ diagonal entries equal to 1 and all other entries equal to 0.

Let $\mat M'$ be the upper-left $k$ by $k$ minor for $\mat D + \gamma\mat E \sqrt{-1} u {u}^*$ and let $w$ be an eigenvector for $\mat M'$ with eigenvalue $\lambda$.  We will also view $w$ as a length $n$ vector by padding with zeros (i.e., defining a new, $n$-dimensional vector also called $w$ by keeping the first $k$ entries and setting entries $k+1$ to $n$ equal to 0); it will be clear from context whether we are viewing $w$ in dimension $k$ or dimension $n$. 

Looking only at the first $k$ columns, the last $n-k$ rows of $\mat D + \gamma\mat E \sqrt{-1} u {u}^* $ are all equal to zero.  For $w$ an eigenvector of $\mat M'$ with eigenvalue $\lambda$, we thus see that padding with zeros so $w$ is an $n$-dimensional vector implies that $w$ is also an eigenvector with the eigenvalue $\lambda$ for the matrix
$\mat D + \gamma\mat E \sqrt{-1} u {u}^* $.  We will show that each of the $k$ eigenvalues for $\mat M'$ has non-zero imaginary part with the same sign as $\gamma$.


Because $\mat M'w=\lambda w$, we also have $(\mat D + \gamma\mat E \sqrt{-1} u {u}^*)w = \lambda w$. If $u^* w = 0$, then $w$ is also an eigenvector for $\mat D$, and because the only non-zero entries in $w$ are in the first $k$ coordinates, we must have that $w$ is a standard basis vector with a 1 in coordinate $i$ and zeros everywhere else, where $1 \le i \le k$.  But since each of the first $k$ coordinates of $u$ is non-zero by assumption, we have $u^*w \ne 0$, a contradiction.  Thus, we conclude that $u^*w \ne 0$.

Consider the number $w^*\partinv{\mat  E} \lambda w$ and note that 
\begin{align*}
w^*\partinv{\mat E} \lambda w &= w^*  \partinv{\mat E}(\mat D + \gamma\mat E \sqrt{-1} u {u}^*)w =
w^*  \partinv{\mat E} \mat D w + \gamma \sqrt{-1} w^*\partI u u^* w \\
&=w^*  \partinv{\mat E} \mat D w + \gamma \sqrt{-1} w^*u u^* w \\
&=w^*  \partinv{\mat E} \mat D w + \gamma \sqrt{-1} \abs{u^* w}^2.
\end{align*}
The last line above is a real number plus a non-zero purely imaginary number, and dividing by the positive real number $w^* \partinv{\mat E} w$ shows that the imaginary part of $\lambda$ is $\gamma \sqrt{-1} \abs{u^*w}^2/(w^* \partinv{\mat E} w)$, which has the same sign as $\gamma$.  Thus we have shown that $\mat D + \gamma\mat E \sqrt{-1} u {u}^*$ has $k$ eigenvalues with non-zero imaginary part, each of which has the same sign as $\gamma$.

To show that the remaining $n-k$ eigenvalues match the last $n-k$ diagonal entries in $\mat D$ (which are all real), note that the matrix $\mat D + \gamma\mat E \sqrt{-1} u {u}^* $ has block structure 
$$\begin{pmatrix}
\mat M' & \bullet \\
\mat 0 & \mat D_2
\end{pmatrix},$$
where $\mat 0$ is the $n-k$ by $k$ zero matrix, $\bullet$ is unspecified entries, and $\mat D_2$ is an $n-k$ by $n-k$ diagonal matrix with diagonal entries matching the last $n-k$ entries in $\mat D$.  Because of the block diagonal structure with all lower-triangular entries outside of the blocks equal to zero, we have that 
$p_{\mat D + \gamma\mat E \sqrt{-1} u {u}^*}(z) = p_{\mat M'}(z) p_{\mat D_2}(z)$, 
where $p_{\mat A}(z)$ is the characteristic polynomial for a square matrix $\mat A$; thus, each diagonal entry in $\mat D_2$ is an eigenvalue for $\mat D + \gamma\mat E \sqrt{-1} u {u}^*,$ completing the proof.
\end{proof}

\begin{proof}[Proof of Theorem~\ref{thm:nonreal-nonrandom-diag}]  
Given that $\mat M$ is Hermitian with $k$ distinct eigenvalues, we can find a unitary matrix $\mat U$ such that $\mat U^*\mat M\mat U =\mat D$, a diagonal matrix with the first $k$ diagonal entries distinct and so that the first $k$ columns of $\mat U$ are the eigenvectors $w_1,\dots,w_k$ in order.  Let $\gamma = 1$ if $a_i >0$ for $i=1,\dots,k$, and let $\gamma = -1$ if $a_i < 0$ for $i=1,\dots,k$.  Let $\mat E$ be the diagonal matrix with the first $k$ diagonal entries $\gamma a_1,\dots, \gamma a_k$ (which are positive real numbers), and let the remaining diagonal entries of $\mat E$ be equal to zero.
Finally, let $z =\transp{(z_1,\dots,z_k,0,\dots, 0)}$, and note that $v=\mat U z$ and $u=\mat U\gamma \mat Ez$ from their definitions.

The eigenvalues of $\mat M + \sqrt{-1}uv^*$ are the same as the eigenvalues of
$$\mat U^*(\mat M + \sqrt{-1}uv^*) \mat U =
\mat D + \sqrt{-1} \mat U^* (\mat U\gamma \mat E z )(\mat Uz)^*\mat U = \mat D + \gamma \sqrt{-1} \mat E z z^*$$

Applying Lemma~\ref{lem:nonreal-nonrandom} to $\mat D + \gamma\sqrt{-1} \mat E z z^*$ completes the proof.
\end{proof}

\subsection{Proof of Theorem~\ref{thm:sampcov-nonreal}}
\label{ss:sampcov-nonreal}

We begin with a corollary to Theorem~\ref{thm:nonreal-nonrandom-diag} adapting results on additive perturbations to multiplicative perturbations.

\begin{corollary}\label{cor:multshift}
Let $\mat M$ be an $n \times n$ Hermitian matrix with $k$ distinct, positive eigenvalues $\lambda_1,\dots, \lambda_k$ with corresponding unit eigenvectors $w_1,\dots,w_k$.  
Let $v$ be any linear combination $\sum_{i=1}^k z_i w_i$ where the $z_i$ are complex scalers and $z_i \ne 0$ for $i=1,\dots, k$.
Then 
\begin{itemize}
\item if $\gamma>0$, $\mat M(\mat I + \gamma \sqrt{-1} v v^*) $ has exactly $k$ eigenvalues with positive imaginary part, and 
\item if $\gamma <0$, $\mat M(\mat I + \gamma \sqrt{-1}v v^*) $  has at exactly $k$ eigenvalues with negative imaginary part.
\end{itemize}
Furthermore, regardless of the value of $\gamma \ne 0$, the matrix $\mat M(\mat I + \gamma \sqrt{-1}v v^*) $ shares $n-k$ eigenvalues with $\mat M$.
\end{corollary}

\begin{proof}
Since $\mat M$ is Hermitian, there exists an unitary matrix $\mat U$ such that $\mat U^* \mat M \mat U = \mat D$, where $\mat D = \diag(\lambda_1,\dots,\lambda_n)$, where $\lambda_1,\dots,\lambda_k$ are distinct, non-zero eigenvalues, and where the columns of $\mat U$ are $w_1,\dots, w_n$.  Note that $v=\mat U z$ where 
$z=
\transp{(z_1,\dots,z_k,0,\dots, 0)}$.

Conjugating by $\mat U^*$, we see that the eigenvalues of $\mat M(\mat I + \gamma \sqrt{-1} v v^*) $ are the same as the eigenvalues of $\mat D + \gamma\mat D \sqrt{-1} zz^* = \mat D + \sqrt{-1}(\gamma\mat D z)z^*$.  We can now apply Theorem~\ref{thm:nonreal-nonrandom-diag}, noting that the eigenvectors corresponding to the $k$ distinct, positive eigenvalues of $\mat D$ are standard basis vectors and setting $v=z$ and $u=\gamma \mat D z$. \end{proof}

To prove Theorem~\ref{thm:sampcov-nonreal}, we will make use of the following lemma, which characterizes some almost sure properties of the eigenvalues and eigenvectors of $\mat S_n$.

\begin{lemma} \label{lemma:distinct}
Let $m$ and $n$ be positive integers, and set $r:=\min\{m,n\}$.  If $\mat S_n$
is a sample covariance matrix with atom variable $\xi$ and parameters $(m,n)$, where $\xi$ is an absolutely continuous random variable, then
\begin{itemize}
\item with probability 1, the matrix $\mat S_n$ has $r$ distinct, strictly positive eigenvalues $\lambda_1, \dots, \lambda_r$, and
\item with probability 1, any eigenvector $v=(v_1,\dots, v_n)$ for $\mat S_n$ corresponding to one of the non-zero eigenvalues $\lambda_1,\dots, \lambda_r$ has all non-zero coordinates, i.e. $v_i \ne 0$ for all $1\le i\le n$. 
\end{itemize}
\end{lemma}

The proof of Lemma~\ref{lemma:distinct} is given in Appendix~\ref{app:distinct}.

\begin{proof}[Proof of Theorem~\ref{thm:sampcov-nonreal}]
The matrix $\mat S_n$ is Hermitian, so there exists a unitary matrix $\mat U$ such that $\mat U^* \mat S_n \mat U= \mat D$, a diagonal matrix with real diagonal entries in decreasing order.  Because $\mat S_n$ is a random sample covariance matrix, we know with probability 1 from Lemma~\ref{lemma:distinct} that $\mat D$ has the first $r=\min\{m,n\}$ diagonal entries $\lambda_1 >\lambda_2 > \dots >\lambda_r > 0$ distinct and strictly positive, and all remaining diagonal entries, if any, equal to zero (which follows by considering the rank of $\mat S_n$).  Let $v$ be a standard basis vector  and let $\{w_i\}_{i=1}^n$ be the columns of $\mat U$, which form an orthonormal basis of eigenvectors for $\mat S_n$; thus, we can write $v = \sum_{i=1}^n z_i w_i$ where each complex constant $z_i \ne 0$ for $1 \le i \le r$ with probability 1 (again by Lemma~\ref{lemma:distinct}).

Consider the matrix $\mat S_n(\mat I +\gamma \sqrt{-1} v v^*)$ and note that it has the same eigenvalues as 
$$\mat U^*\mat S_n(\mat I +\gamma \sqrt{-1} v v^*)\mat U = \mat D+\gamma\mat D \sqrt{-1} \mat U^* v v^* \mat U$$ 
We may now apply Lemma~\ref{lem:nonreal-nonrandom} with $\mat E = \mat D$ and $k=r$, noting that $u=\mat U^* v$ is a vector with all coordinates non-zero with probability 1 and that the conditions on $\mat D$ and $\mat E$ are also satisfied with probability 1.
\end{proof}

\section{Proof of results in Section \ref{sec:global}} \label{sec:proof:global}

This section is devoted to the proof of Theorem \ref{thm:global:sc} and Theorem \ref{thm:global:mp}.  We will use Theorem \ref{thm:kahan} to show that the empirical spectral measure of a Hermitian matrix does not significantly change when perturbed by a arbitrary matrix $\mat P$.  Indeed, both theorems will follow from the deterministic lemma below.  

\begin{lemma} \label{lemma:lip}
Let $\mat{M}$ be an $n \times n$ Hermitian matrix, and assume $\mat{P}$ is an arbitrary $n \times n$ matrix.  Then for any bounded, Lipschitz function $f:\mathbb{C} \to \mathbb{C}$, there exists a constant $C > 0$ (depending only on the Lipschitz constant of $f$) such that 
$$ \left| \int f d \mu_{\mat{M}} - \int f d \mu_{\mat M + \mat P} \right| \leq \frac{C}{\sqrt{n}} \| \mat P \|_2. $$
\end{lemma}
\begin{proof}
Let $f: \mathbb{C} \to \mathbb{C}$ be a bounded, Lipschitz function with Lipschitz constant $C_f$.  Let $\lambda_1 \geq \cdots \geq \lambda_n$ denote the ordered eigenvalues of $\mat M$.  Let $\alpha_1 + \sqrt{-1} \beta_1, \ldots, \alpha_n + \sqrt{-1} \beta_n$ be the eigenvalues of $\mat{M} + \mat P$ ordered so that $\alpha_1 \geq \cdots \geq \alpha_n$.  Then, by the Cauchy-Schwarz inequality and Theorem \ref{thm:kahan}, we obtain
\begin{align*}
	\left| \int f d \mu_{\mat{M}} - \int f d \mu_{\mat M + \mat P} \right| &= \frac{1}{n} \left| \sum_{i=1}^n f(\lambda_i) - \sum_{i=1}^n f(\alpha_i + \sqrt{-1} \beta_i) \right| \\
		&\leq \frac{C_f}{n} \sum_{i=1}^n \left|\lambda_i - \alpha_i - \sqrt{-1} \beta_i \right| \\
		&\leq \frac{C_f}{\sqrt{n}} \sqrt{ \sum_{i=1}^n \left| \lambda_i - \alpha_i - \sqrt{-1} \beta_i \right|^2 } \\
		&\leq \frac{\sqrt{2} C_f}{\sqrt{n}} \| \mat P \|_2,
\end{align*}
as desired.
\end{proof}
\begin{remark}
One can also show that 
$$ L(\mu_{\mat{M}}, \mu_{\mat{M} + \mat{P}})^3 \leq \frac{1}{n} \| \mat P \|_2^2, $$
where $L(\mu, \nu)$ denotes the Levy distance between the probability measures $\mu$ and $\nu$.  See \cite[Theorem A.38]{BSbook} and \cite[Remark A.39]{BSbook} for details.  
\end{remark}

We now prove Theorem \ref{thm:global:sc}.  

\begin{proof}[Proof of Theorem \ref{thm:global:sc}]
By the portmanteau theorem (see, for instance, \cite[Theorem 11.3.3]{D}), it suffices to show that, a.s., 
$$ \lim_{n \to \infty} \int f d \mu_{\frac{1}{\sqrt{n}} \mat W_n + \mat P_n} = \int f d \mu_{\mathrm{sc}} $$
for every bounded, Lipschitz function $f: \mathbb{C} \to \mathbb{C}$.  

Theorem \ref{thm:wigner} implies that, almost surely, 
$$ \lim_{n \to \infty} \int f d \mu_{\frac{1}{\sqrt{n}} \mat W_n} = \int f d \mu_\mathrm{sc} $$
for every bounded, Lipschitz function $f: \mathbb{R} \to \mathbb{R}$ (and hence for every bounded, Lipschitz function $f: \mathbb{C} \to \mathbb{C}$).  Thus, by the triangle inequality, it suffices to show
$$ \lim_{n \to \infty} \left| \int f d \mu_{\frac{1}{\sqrt{n}} \mat W_n} - \int f d \mu_{\frac{1}{\sqrt{n}} \mat W_n + \mat P_n} \right| = 0. $$
The claim now follows from Lemma \ref{lemma:lip} and condition \eqref{eq:Pnhs}.   
\end{proof}

In order to prove Theorem \ref{thm:global:mp}, we will need the following well-known bound on the spectral norm of a sample covariance matrix.

\begin{theorem}[Spectral norm of a sample covariance matrix; Theorem 5.8 from \cite{BSbook}] \label{thm:sc:norm}
Let $\xi$ be a real random variable with mean zero, unit variance, and finite fourth moment.  For each $n \geq 1$, let $\mat S_n$ be a sample covariance matrix with atom variable $\xi$ and parameters $(m,n)$, where $m$ is a function of $n$ such that $y_n := \frac{m}{n} \to y \in (0, \infty)$ as $n \to \infty$.  Then a.s.
$$ \limsup_{n \to \infty} \frac{1}{\sqrt{m n}} \| \mat S_n \| \leq \lambda_+, $$
where $\lambda_+$ is defined in \eqref{eq:def:lambdapm}.  
\end{theorem}

\begin{proof}[Proof of Theorem \ref{thm:global:mp}]
Since $\mat S_n(\mat I + \mat P) = \mat S_n + \mat S_n \mat P$, it suffices, by Theorem~\ref{thm:mp} and Lemma \ref{lemma:lip}, to show that a.s. $\frac1{n\sqrt m} \|\mat S_n \mat P\|_2\to 0$ as $n\to \infty$. 
From \cite[Theorem A.10]{BSbook}, it follows that
$$ \| \mat S_n \mat P_n \|_2 \leq \| \mat S_n \| \|\mat P_n \|_2. $$
Hence, the claim follows from Theorem \ref{thm:sc:norm} and condition \eqref{eq:Pnhs}.  
\end{proof}

\section{Proof of results in Section \ref{sec:refine}} \label{sec:proof:refine}

This section is devoted to the results in Section \ref{sec:refine}.  We begin with the following deterministic lemma.  

\begin{lemma}[Eigenvalue criterion] \label{lemma:eigcriterion}
Let $\mat M$ and $\mat P$ be arbitrary $n \times n$ matrices.  Suppose $z \in \mathbb{C}$ is not an eigenvalue of $\mat M$.  Then $z$ is an eigenvalue of $\mat M + \mat P$ if and only if 
$$ \det(\mat I + (\mat M - z \mat I)^{-1} \mat P) = 0. $$
\end{lemma}
\begin{proof}
Suppose $z \in \mathbb{C}$ is not an eigenvalue of $\mat M$.  Then $z$ is an eigenvalue of $\mat M + \mat P$ if and only if 
$$ \det (\mat M + \mat P - z \mat I) = 0. $$
This is equivalent to the condition that
\begin{equation} \label{eq:condeig}
	\det (\mat M - z \mat I) \det ( \mat I + (\mat M - z \mat I)^{-1} \mat P) = 0. 
\end{equation}
Since $z$ is not an eigenvalue of $\mat M$, $\det (\mat M - z \mat I) \neq 0$.  Thus, condition \eqref{eq:condeig} is equivalent to 
$$ \det( \mat I + (\mat M - z \mat I)^{-1}\mat P ) = 0, $$
as desired.
\end{proof}

Versions of Lemma \ref{lemma:eigcriterion} have also appeared in \cite{AGG, BY, BGGM, BGN, CDMF, KY, KY2, OR, PRS, Tout}.  

We will also make use of the following identity:
\begin{equation} \label{eq:fundamental}
	\det \left( \mat{I} + \mat {A} \mat{B} \right) = \det \left( \mat{I} + \mat{B} \mat{A} \right),
\end{equation}
valid for arbitrary $n \times k$ matrices $\mat{A}$ and $k \times n$ matrices $\mat{B}$.  We observe that the left-hand side is an $n \times n$ determinant while the right-hand side is a $k \times k$ determinant.  For low-rank perturbations, we will apply \eqref{eq:fundamental} with $k$ fixed and $n$ tending to infinity to transform an unbounded-dimensional problem to a finite-dimensional one.  

\subsection{Proof of Theorem \ref{thm:wigner:refine}}
We now focus on the proof of Theorem \ref{thm:wigner:refine}.  We will need the following concentration result for the largest eigenvalues of a Wigner matrix.  

\begin{lemma}[Bounds on the largest eigenvalue \cite{EYY,AEKYY,EKYY}] \label{lemma:wigner:largest}
Let $\xi$ and $\zeta$ be real random variables which satisfy condition {\bf C0}.  For each $n \geq 1$, let $\mat{W}_n$ be an $n \times n$ Wigner matrix with atom variables $\xi$ and  $\zeta$.  Then, for any $\eps > 0$, a.s., for $n$ sufficiently large, 
$$ \sup_{1 \leq i \leq n} \left| \frac{1}{\sqrt{n}} \lambda_i(\mat W_n) \right| \leq 2 + n^{-2/3 + \eps}. $$
\end{lemma}

Lemma \ref{lemma:wigner:largest} follows from \cite[Theorem 2.1]{EYY} (see also \cite[Theorem 3.7]{KY}) for condition {\bf C0} (ii) by applying the Borel--Cantelli lemma; and the same conclusion follows under condition {\bf C0} (i) using \cite[equation (2.33)]{AEKYY} (which follows from \cite[Theorem~7.6]{EKYY}) and the Borel--Cantelli lemma.  

We define $m_{\mathrm{sc}}$ to be the Stieltjes transform of the semicircle law.  That is,
$$ m_{\mathrm{sc}}(z) := \int_{-\infty}^\infty \frac{\rho_{\mathrm{sc}}(x)}{x - z} dx $$
for $z \in \mathbb{C}$ with $z \not\in [-2,2]$.  In addition, $m_{\mathrm{sc}}$ is also characterized as the unique solution of
\begin{equation} \label{eq:mscdef}
	m_{\mathrm{sc}}(z) + \frac{1}{m_{\mathrm{sc}}(z)} + z = 0 
\end{equation}
that satisfies $\Im m_{\mathrm{sc}}(z) > 0$ when $\Im(z) > 0$, extended where possible by analyticity.

\begin{lemma} \label{lemma:msc}
For all $z \in \mathbb{C}$ with $z \not\in [-2,2]$, $|m_{\mathrm{sc}}(z)| \leq 1$. 
\end{lemma}
\begin{proof}
The case when $\Im z > 0$ follows from \cite[Lemma 3.4]{EYY}.  By symmetry, we also obtain that $|m_{\mathrm{sc}}(z)| \leq 1$ when $\Im z < 0$.  Thus, it suffices to consider the case when $z$ is real.  A simple continuity argument verifies the same bound on $\mathbb{R} \setminus [-2,2]$.  
\end{proof}

The key tool we will need to prove Theorem \ref{thm:wigner:refine} is the following isotropic semicircle law from \cite{AEKYY,KY}.  

\begin{theorem}[Isotropic semicircle law] \label{thm:wigner:isotropic}
Let $\xi$ and $\zeta$ be real random variables which satisfy condition {\bf C0}.  For each $n \geq 1$, let $\mat{W}_n$ be an $n \times n$ Wigner matrix with atom variables $\xi$ and  $\zeta$, and let $u_n$ and $v_n$ be unit vectors in $\mathbb{C}^n$.  Then, for any $\kappa > 3$ and $\eps > 0$, a.s.
\begin{equation}\label{eq:wigner:isotropic}
 \lim_{n \to \infty} \sup_{z \in \mathcal{S}_{n}(\kappa, \eps)} \left| u_n^\ast \left( \frac{1}{\sqrt{n}} \mat W_n - z \mat I \right)^{-1} v_n - m_{\mathrm{sc}}(z) u_n^\ast v_n \right| = 0, 
 \end{equation}
where $\mathcal{S}_n(\kappa, \eps)$ is the union of
$$ \left\{ z = E + \sqrt{-1} \eta : |E| \leq \kappa\ ,\  n^{-1 + \eps} \leq |\eta| \leq \kappa \right\} $$
and 
\begin{equation} \label{eq:wigner:outside}
	\left\{ z = E + \sqrt{-1} \eta : 2 + n^{-2/3 + \eps} \leq |E| \leq \kappa\ ,\  |\eta| \leq \kappa \right\}. 
\end{equation}
\end{theorem}

Theorem \ref{thm:wigner:isotropic} essentially follows from \cite[Theorems 2.12 and 2.15]{AEKYY} and \cite[Theorems 2.2 and 2.3]{KY}.  We require both sets of theorems because Definition \ref{def:C0} (condition {\bf C0}) has two separate (non-equivalent) conditions.  Indeed, the results in \cite{AEKYY} hold under the first set of conditions in Definition \ref{def:C0} and those in \cite{KY} hold under the second set.  

Theorem \ref{thm:wigner:isotropic} follows from the results in \cite{AEKYY,KY} with the following modifications.
\begin{itemize}
\item The results in \cite{AEKYY,KY} only deal with points $z \in \mathcal{S}_{n}(\kappa,\eps) \cap \mathbb{C}^+$ (i.e. points in the upper half-plane).  Since the entries of $\mat{W}_n$ are real, we have
$$ \overline{ \left( \frac{1}{\sqrt{n}} \mat W_n - z \mat I \right)^{-1} } = \left( \frac{1}{\sqrt{n}} \mat W_n - \overline{z} \mat I \right)^{-1}, $$
and hence the results also easily extend to the lower half-plane.

\item The results in \cite{AEKYY,KY} also only hold for a fixed point $z \in \mathcal{S}_n(\kappa, \eps)$.  However, this can easily be overcome by a simple net argument; see \cite[Remark 2.6]{AEKYY} and \cite[Remark 2.4]{KY} for details.
\item In \cite[Theorem 2.15]{AEKYY} and \cite[Theorem 2.3]{KY}, the authors consider a region similar to \eqref{eq:wigner:outside} but missing the case when $\eta = 0$ (i.e. the part of the region in \eqref{eq:wigner:outside} that intersects the real line).  This issue can be remedied, however, by using the net argument of the previous bullet point along with the facts below that ensure that real line is close enough to where \cite[Theorem 2.15]{AEKYY} and \cite[Theorem 2.3]{KY} hold with sufficient strength.  Indeed, due to the fact that $\Im(m_{\mathrm{sc}}(z))$ becomes nearly proportional to $\Im(z)$ (see \cite[Lemma~7.1]{AEKYY}), the error bounds in \cite[Theorem 2.15]{AEKYY} and \cite[Theorem 2.3]{KY} can be bounded by a term that does not depend on $\eta$.  This, along with the fact that that $m_{\mathrm{sc}}$ is Lipschitz continuous with constant at most $5n^{2/3}$ on $\mathcal{S}_n(\kappa, \eps)$ shows that the net argument extends to include the real line (see also \cite[Remark~2.7]{AEKYY} for similar reasoning in the sample covariance case).

\item Finally, the results in \cite{AEKYY, KY} hold with very high probability.  To achieve the almost sure bounds in Theorem \ref{thm:wigner:isotropic}, one can apply the Borel--Cantelli lemma.  

\end{itemize}

With Theorem \ref{thm:wigner:isotropic} in hand, we can now prove Theorem \ref{thm:wigner:refine}.  

\begin{proof}[Proof of Theorem \ref{thm:wigner:refine}]
Take $\kappa := \sup_{n \geq 1} \| \mat P_n \| + 3$.  By Lemma \ref{lemma:wigner:largest}, a.s., for $n$ sufficiently large, 
$$ \left\| \frac{1}{\sqrt{n}} \mat W_n + \mat P_n \right\| \leq \kappa, $$
and hence all eigenvalues of $\frac{1}{\sqrt{n}} \mat W_n + \mat P_n$ are contained in the region
$$ \{z \in \mathbb{C} : |z| \leq \kappa \} \subseteq \{z = E + \sqrt{-1} \eta : |E| \leq \kappa, |\eta| \leq \kappa \}. $$
In addition, by Lemma \ref{lemma:wigner:largest} (and the fact that the eigenvalues of $\mat W_n$ are real), a.s., for $n$ sufficiently large, no eigenvalue of $\frac{1}{\sqrt{n}} \mat W_n$ can be in $\mathcal{S}_n(\kappa,\eps)$, where $\mathcal{S}_n(\kappa,\eps)$ is defined in Theorem \ref{thm:wigner:isotropic}.  
Thus, in view of Lemma \ref{lemma:eigcriterion}, we find that $z \in \mathcal{S}_n(\kappa, \eps)$ is an eigenvalue of $\frac{1}{\sqrt{n}} \mat W_n + \mat P_n$ if and only if 
$$ \det(\mat I + \mat G_n(z) \mat P_n) = 0, $$
where 
$$ \mat G_n(z) := \left( \frac{1}{\sqrt{n}} \mat W_n - z \mat I \right)^{-1} $$
is the resolvent of $\frac{1}{\sqrt{n}} \mat W_n$.  By the singular value decomposition, we can write $\mat P_n = \mat A_n \mat B_n$ for some $n \times k$ and $k \times n$ matrices $\mat {A}_n$ and $\mat B_n$, both of operator norm $O(1)$.  Thus, by \eqref{eq:fundamental}, it follows that  $z \in \mathcal{S}_n(\kappa, \eps)$ is an eigenvalue of $\frac{1}{\sqrt{n}} \mat W_n + \mat P_n$ if and only if 
\begin{equation} \label{eq:detBA}
	\det ( \mat I + \mat B_n \mat G_n(z) \mat A_n) = 0. 
\end{equation}
Observe that the determinant in \eqref{eq:detBA} is a determinant of an $k \times k$ matrix.  Since $k = O(1)$, Theorem \ref{thm:wigner:isotropic} implies that, a.s., 
$$ \det ( \mat I + \mat B_n \mat G_n(z) \mat A_n) = \det(\mat I + m_{\mathrm{sc}}(z) \mat B_n \mat A_n) + o(1) $$
uniformly for $z \in \mathcal{S}_n(\kappa, \eps)$.  By another application of \eqref{eq:fundamental}, we conclude that if $z \in \mathcal{S}_n(\kappa, \eps)$ is an eigenvalue of $\frac{1}{\sqrt{n}} \mat W_n + \mat P_n$, then
\begin{equation} \label{eq:retPn}
	\det (\mat I + m_{\mathrm{sc}}(z) \mat B_n \mat A_n) = \det (\mat I + m_{\mathrm{sc}}(z) \mat P_n) = o(1). 
\end{equation}

We will return to \eqref{eq:retPn} in a moment, but first need to restrict the domain $\mathcal{S}_n(\kappa, \eps)$ slightly.  We recall that the eigenvalues $\lambda_1(\mat P_n), \ldots, \lambda_j(\mat P_n)$ satisfy
\begin{equation} \label{eq:geqdelta}
	|\lambda_i(\mat P_n)| \geq 1 + \delta, 
\end{equation}
and the remaining nontrivial eigenvalues of $\mat P_n$ (some of which may be zero) satisfy 
\begin{equation} \label{eq:leqdelta}
	|\lambda_i(\mat P_n)| \leq 1 - \delta 
\end{equation}
for $j+1 \leq i \leq k$.  By Lemma \ref{lem:ellipse} below, there exists $\delta' > 0$ such that $\delta'<\frac{\delta^2}{1+\delta}$ and 
\begin{equation} \label{eq:3delta'}
	\lambda_i(\mat P_n) + \frac{1}{\lambda_i(\mat P_n)} \not\in \Esc_{3\delta'} \qquad \text{for } 1 \leq i \leq j.
\end{equation}
Hence the assumptions of Theorem \ref{thm:or} are satisfied.  By Theorem \ref{thm:or}, a.s., for $n$ sufficiently large, there are exactly $j$ eigenvalues of $\frac{1}{\sqrt{n}} \mat W_n + \mat P_n$ in the region $\mathbb{C} \setminus \Esc_{2 \delta'}$ and after labeling the eigenvalues properly
$$ \lambda_i \left( \frac{1}{\sqrt{n}} \mat W_n + \mat P_n \right) = \lambda_i (\mat P_n) + \frac{1}{ \lambda_i(\mat P_n)} + o(1) $$
for $1 \leq i \leq j$.  Thus, it suffices to show that a.s.~the remaining $n-j$ eigenvalues of $\frac{1}{\sqrt{n}} \mat W_n + \mat P_n$ are not contained in 
$$ \tilde{\mathcal{S}}_n(\kappa, \eps) := \mathcal{S}_n(\kappa, \eps) \cap \Esc_{2\delta'}. $$  

Returning to \eqref{eq:retPn}, we see that a.s., for $n$ sufficiently large, if $z \in \tilde{\mathcal{S}}_n(\kappa, \eps)$ is an eigenvalue of $\frac{1}{\sqrt{n}} \mat W_n + \mat P_n$, then 
\begin{equation} \label{eq:detprod}
	\det (\mat I + m_{\mathrm{sc}}(z) \mat P_n) = \prod_{i=1}^k (1 + m_{\mathrm{sc}}(z) \lambda_i(\mat P_n)) = o(1). 
\end{equation}
In order to reach a contradiction, suppose $z \in \tilde{\mathcal{S}}_n(\kappa, \eps)$ is an eigenvalue of $\frac{1}{\sqrt{n}} \mat W_n + \mat P_n$ and 
that 
\eqref{eq:detprod} holds.  Then there exists $1 \leq i \leq k$ such that
\begin{equation} \label{eq:loc1msc}
	1 + m_{\mathrm{sc}}(z) \lambda_i(\mat P_n) = o(1). 
\end{equation}
There are two cases to consider: either $\lambda_i(\mat P_n)$ satisfies \eqref{eq:geqdelta} or \eqref{eq:leqdelta}.  However, in view of Lemma \ref{lemma:msc}, it must be the case that \eqref{eq:geqdelta} holds, and hence $1 \leq i \leq j$.  Thus, from \eqref{eq:loc1msc}, we find that 
$$ m_{\mathrm{sc}}(z) = \frac{-1}{\lambda_i(\mat P_n)} + o(1). $$
Using \eqref{eq:mscdef}, this yields that
$$ z = \lambda_i(\mat P_n) + \frac{1}{\lambda_i(\mat P_n)} + o(1). $$
From \eqref{eq:3delta'}, for $n$ sufficiently large, this implies that $z \not\in \Esc_{2\delta'}$, a contradiction of our assumption that $z \in \tilde{\mathcal{S}}_n(\kappa, \eps)$.  We conclude that, a.s., for $n$ sufficiently large, no eigenvalues of $\frac{1}{\sqrt{n}} \mat W_n + \mat P_n$ are contained in $\tilde{\mathcal{S}}_n(\kappa, \eps)$, and the proof is complete.  
\end{proof}

It remains to prove the following lemma.  

\begin{lemma} \label{lem:ellipse}
Suppose $\lambda \in \mathbb{C}$ with $|\lambda| \geq 1 + \delta$ for some $\delta > 0$.  Then there exists $\delta' > 0$ (depending only on $\delta$) such that
$$ \lambda + \frac{1}{\lambda} \not\in \Esc_{\delta'}. $$
Furthermore, one can ensure that $\delta'<\frac{\delta^2}{1+\delta}$.
\end{lemma}
\begin{proof}
For any $r > 1$, let $E_r$ denote the ellipse
$$ E_r := \left\{ z \in \mathbb{C} : \frac{\Re(z)^2}{\left( r + \frac{1}{r} \right)^2} + \frac{\Im(z)^2}{\left( r - \frac{1}{r} \right)^2} = 1 \right\}. $$ 
Let $\dist(E_r, [-2,2])$ denote the distance in the complex plane between the ellipse $E_r$ and the line segment $[-2,2] \subset \mathbb{R}$.

Assume $\lambda \in \mathbb{C}$ with $|\lambda| \geq 1 + \delta$.  We write $\lambda$ in polar coordinates as $\lambda = r e^{\sqrt{-1} \theta}$, where $r \geq 1 + \delta$.  Thus, 
$$ \lambda + \frac{1}{\lambda} = \left ( r + \frac{1}{r} \right) \cos \theta + \sqrt{-1} \left( r - \frac{1}{r} \right) \sin \theta. $$
We conclude that $\lambda + \frac{1}{\lambda} \in E_{r}$.  We also note that, for $r > 1$, the function $r \mapsto r + 1/r$ 
is strictly increasing. This shows that $E_r$ and the segment $[-2,2]$ are disjoint, since $r+\frac1r>2$.

It follows that $E_{r}$ and the line segment $[-2,2]$ are compact disjoint subsets in the complex plane.  Hence, $\dist(E_{r},[-2, 2]) > 0$, and the first claim follows.  

To show that $\delta'<\frac{\delta^2}{1+\delta}$, we note that, because $\delta'$ depends only on $\delta$, we may choose $\lambda=1+\delta$.  Thus, we know that $\lambda+\frac 1\lambda -2= 1+\delta +\frac1{1+\delta}-2 > \delta'$, an inequality that can be rearranged to show $\delta'<\frac{\delta^2}{1+\delta}$.
\end{proof}

\subsection{Proof of Theorem \ref{thm:sc:refine}}
We now turn our attention to the the proof of Theorem \ref{thm:sc:refine}.  Let $m_{\mathrm{MP}}$ be the Stieltjes transform of the Marchenko--Pastur law $\mu_{\mathrm{MP},1}$.  That is, 
$$ m_{\mathrm{MP}}(z) := \int_{\mathbb{R}} \frac{\rho_{\mathrm{MP},1}(x)}{x - z} dx $$
for $z \in \mathbb{C}$ with $z \not\in [0,4]$, where $\rho_{\mathrm{MP},1}$ is the density defined in \eqref{eq:def:rhoMP} and \eqref{eq:def:lambdapm} with $y = 1$.  In particular, $m_{\mathrm{MP}}$ is characterized as the unique solution of 
\begin{equation} \label{eq:mMP}
	m_{\mathrm{MP}}(z) + \frac{1}{z + z m_{\mathrm{MP}}(z)} = 0 
\end{equation}
satisfying $\Im m_{\mathrm{MP}}(z) > 0$ for $\Im z > 0$ and extended where possible by analyticity.   In addition, $m_{\mathrm{MP}}$ satisfies the following bound.  

\begin{lemma} \label{lemma:mMP}
For all $z \in \mathbb{C}$ with $z \not \in [0,4]$, $|1 + z m_{\mathrm{MP}}(z)| \leq 1$.  
\end{lemma}
\begin{proof}
We observe that, for $z \in \mathbb{C}$ with $z \not\in [-2,2]$, we have $z m_{\mathrm{MP}}(z^2) = m_{\mathrm{sc}}(z)$.  Thus,
$$ |1 + z^2 m_{\mathrm{MP}}(z^2)| = |1 + z m_{\mathrm{sc}}(z)| = |m_{\mathrm{sc}}(z)|^2 \leq 1 $$
by \eqref{eq:mscdef} and Lemma \ref{lemma:msc}.  
\end{proof}

We now state the analogue of Theorem \ref{thm:wigner:isotropic}.  

\begin{theorem}[Isotropic Marchenko--Pastur law] \label{thm:sc:isotropic}
Let $\xi$ be a real random variable which satisfies condition {\bf C1}.  For each $n \geq 1$, let $\mat{S}_n$ be an $n \times n$ sample covariance matrix with atom variable $\xi$ and parameters $(n,n)$, and let $u_n$ and $v_n$ be unit vectors in $\mathbb{C}^n$.  Then, for any $\kappa > 4$ and $\delta', \eps > 0$, a.s.
$$ \lim_{n \to \infty} \sup_{z \in \mathcal{T}_n(\kappa, \delta',\eps)} \left| u_n^\ast \left( \frac{1}{n} \mat S_n - z \mat I \right)^{-1} v_n - m_{\mathrm{MP}}(z) u_n^\ast v_n \right| = 0, $$
where $\mathcal{T}_n(\kappa, \delta',\eps)$ is the union of 
$$ \left \{ z = E + \sqrt{-1} \eta : |z| \leq \kappa, n^{-1 + \eps} \leq |\eta| \leq \kappa, |z| \geq \delta' \right\}, $$
\begin{equation} \label{eq:sc:outside1}
	\left \{ z = E + \sqrt{-1} \eta : 4 + n^{-2/3 + \eps} \leq E \leq \kappa, |\eta| \leq \kappa \right\},
\end{equation}
and
\begin{equation} \label{eq:sc:outside2} 
	\left\{ z = E + \sqrt{-1} \eta : E < 0, |E| \leq \kappa, |\eta| \leq \kappa, |z| \geq \delta' \right\}. 
\end{equation}
\end{theorem}

Theorem \ref{thm:sc:isotropic} follows from \cite[Theorems 2.4 and 2.5]{AEKYY} with the following modifications.
\begin{itemize}
\item The results in \cite{AEKYY} only deal with points $z \in \mathcal{T}_{n}(\kappa,\delta',\eps) \cap \mathbb{C}^+$ (i.e. points in the upper half-plane).  The results easily extend to the lower half-plane by symmetry.  
\item The results in \cite{AEKYY} also only hold for a fixed point $z \in \mathcal{S}_n(\kappa, \eps)$.  However, this can easily be overcome by a simple net argument; see \cite[Remark 2.6]{AEKYY} for details.
\item The case when $\eta = 0$ in \eqref{eq:sc:outside1} and \eqref{eq:sc:outside2} can again be obtained by a continuity argument; see \cite[Remark 2.7]{AEKYY} for details.   
\item The results in \cite{AEKYY} hold with very high probability.  To achieve the almost sure bounds in Theorem \ref{thm:wigner:isotropic}, one must apply the Borel--Cantelli lemma. 
\end{itemize}

We will need the following bound on the largest eigenvalue of a sample covariance matrix, which is the analogue of Lemma \ref{lemma:wigner:largest}.  

\begin{lemma}[Bounds on the largest eigenvalue] \label{lemma:sc:largest}
Let $\xi$ be a real random variable which satisfies condition {\bf C1}.  For each $n \geq 1$, let $\mat S_n$ be a sample covariance matrix with atom variable $\xi$ and parameters $(n,n)$.  Then for any $\eps > 0$, a.s., for $n$ sufficiently large, 
$$ \sup_{1 \leq i \leq n } \left| \frac{1}{n} \lambda_i(\mat S_n) \right| \leq 4 + n^{-2/3 + \eps}. $$
\end{lemma}

Lemma \ref{lemma:sc:largest} follows from \cite[Theorem 2.10]{AEKYY}.  We will also need the following result which is the analogue of Lemma \ref{lem:ellipse}.

\begin{lemma} \label{lem:sc}
Suppose $\lambda \in \mathbb{C}$ with $|\lambda| \geq 1 + \delta$ for some $\delta > 0$.  Then there exists $\delta' > 0$ (depending only on $\delta$) such that
$$ \lambda \left( 1 + \frac{1}{\lambda} \right)^2 \not\in \Emp_{\delta'}. $$
Furthermore, one can ensure that $\delta'<\frac{\delta^2}{1+\delta}$.
\end{lemma}
\begin{proof}
We observe that
$$ \lambda \left( 1 + \frac{1}{\lambda} \right)^2 = (\lambda + 1) \left( 1 + \frac{1}{\lambda} \right) = 2 + \lambda + \frac{1}{\lambda}. $$
By Lemma \ref{lem:ellipse}, there exists $\delta' > 0$ such that $\lambda + 1/\lambda \not\in \Esc_{\delta'}$ and $\delta'<\frac{\delta^2}{1+\delta}$.  Hence, we conclude that $2 + \lambda + \frac{1}{\lambda} \not\in \Emp_{\delta'}$.  
\end{proof}

\begin{proof}[Proof of Theorem \ref{thm:sc:refine}]
Let $\lambda_1(\mat P_n), \ldots, \lambda_j(\mat P_n)$ be the $j$ eigenvalues of $\mat P_n$ which satisfy \eqref{eq:sc:largeevalues}.  Then, by Lemma~\ref{lem:sc}, there exists $\delta' > 0$ such that $\delta'<\frac{\delta^2}{1+\delta}$ and 
\begin{equation} \label{eq:sc:3delta}
	\lambda_i(\mat P_n) \left( 1 + \frac{1}{\lambda_i(\mat P_n)} \right)^2 \not\in \Emp_{3\delta'} 
\end{equation}
for $1 \leq i \leq j$.  Take 
$$ \kappa := 5 + 5\sup_{n \geq 1} \| \mat P_n \|. $$ 
By Theorem \ref{thm:sc:norm}, a.s., for $n$ sufficiently large, $\left\| \frac{1}{n} \mat S_n (\mat I+ \mat P_n) \right\| \leq \kappa$.  Hence, a.s., for $n$ sufficiently large, all the eigenvalues of $\frac{1}{n} \mat S_n (\mat I+ \mat P_n)$ are contained in the disk $\{z \in \mathbb{C} : |z| \leq \kappa \}$.  

We first consider the region 
$$ \mathcal{T}'_n(\kappa, \delta') := \left\{ z \in \mathbb{C} : |z| \leq \kappa \right\} \cap \left( \mathbb{C} \setminus \Emp_{\delta'} \right). $$
Notice that, by Theorem \ref{thm:sc:norm} (and the fact that the eigenvalues of $\frac{1}{n} \mat S_n$ are real and non-negative), a.s., for $n$ sufficiently large, no point outside $\Emp_{\delta'}$ is an eigenvalue of $\frac{1}{n} \mat S_n$.  By Lemma \ref{lemma:eigcriterion}, a.s., for $n$ sufficiently large, $z \in \mathcal{T}'_n(\kappa, \delta')$ is an eigenvalue of $\frac{1}{n} \mat S_n (\mat I+ \mat P_n)$ if and only if 
$$ f(z) := \det \left( \mat I + \mat R_n(z) \frac{1}{n} \mat S_n \mat P_n \right) = 0, $$
where 
$$ \mat R_n(z) := \left( \frac{1}{n} \mat S_n - z \mat I \right)^{-1} $$
is the resolvent of $\frac{1}{n} \mat S_n$.  Rewriting, we find that
$$ f(z) = \det \left( \mat I + \left( \mat I + z \mat R_n(z) \right) \mat P_n \right). $$
By the singular value decomposition, we can write $\mat P_n = \mat A_n \mat B_n$, where $\mat A_n$ is an $n \times k$ matrix, $\mat B_n$ is a $k \times n$ matrix, and both $\mat A_n$ and $\mat B_n$ have spectral norm $O(1)$.  Thus, by \eqref{eq:fundamental}, we obtain that
$$ f(z) = \det \left( \mat I + \mat B_n \left( \mat I + z \mat R_n(z) \right) \mat A_n \right). $$

Introduce the function 
$$ g(z) := \prod_{i=1}^k \left( 1 + (1 + z m_{\mathrm{MP}}(z)) \lambda_j(\mat P_n) \right), $$
where $\lambda_1(\mat P_n), \ldots, \lambda_k(\mat P_n)$ are the non-trivial eigenvalues of $\mat P_n$ (some of which may be zero), including the $j$ eigenvalues which satisfy \eqref{eq:sc:largeevalues}.   Then $f$ and $g$ are both meromorphic functions that are equal to $1$ at infinity.  In view of Lemma \ref{lemma:eigcriterion}, it follows that a.s., for $n$ sufficiently large, the zeros of $f(z)$ correspond to eigenvalues of $\frac{1}{n} \mat S_n (\mat I+ \mat P_n)$ outside $\Emp_{\delta'}$.  In addition, an inspection of the argument of Lemma \ref{lemma:eigcriterion} reveals that the multiplicity of a given eigenvalue is equal to the degree of the corresponding zero of $f$. 

From Theorem \ref{thm:sc:isotropic} and another application of \eqref{eq:fundamental}, we have that a.s., 
\begin{align*}
	f(z) &= \det \left( \mat I + (1 + z m_{\mathrm{MP}}(z)) \mat B_n\mat A_n \right) + o(1) \\
	&= \det \left( \mat I + (1 + z m_{\mathrm{MP}}(z)) \mat P_n \right) + o(1) \\
	&= g(z) + o(1) 
\end{align*}
uniformly for $z \in \mathcal{T}'_n(\kappa, \delta')$.   Thus, by Rouche's theorem (or the argument principle) and Lemma~\ref{lem:sc:ident} below, a.s., there are exactly $j$ eigenvalues of $\frac{1}{n} \mat S_n (\mat I+ \mat P_n)$ not in the region $\Emp_{\delta'}$ and, after labeling the eigenvalues properly, they satisfy
$$ \lambda_i \left( \frac{1}{n} \mat S_n ( \mat I + \mat P_n) \right) = \lambda_i( \mat P_n) \left( 1 + \frac{1}{\lambda_i(\mat P_n)} \right)^2 + o(1) $$
for $1 \leq i \leq j$.  

We now consider the region
$$ \tilde{\mathcal{T}}_n(\kappa, \delta', \eps) := \mathcal{T}_n(\kappa, \delta', \eps) \cap \Emp_{\delta'}, $$
where $\mathcal{T}_n(\kappa, \delta', \eps)$ is defined in Theorem \ref{thm:sc:isotropic}.  It remains to show that a.s., for $n$ sufficiently large, no eigenvalue of $\frac{1}{n} \mat S_n (\mat I + \mat P_n)$ is contained in $\tilde{\mathcal{T}}_n(\kappa, \delta',\eps)$.  

By Lemma \ref{lemma:sc:largest} (and the fact that the eigenvalues of $\frac{1}{n} \mat S_n$ are real and non-negative), it follows that a.s., for $n$ sufficiently large, no point in $\tilde{\mathcal{T}}_n(\kappa, \delta', \eps)$ is an eigenvalue of $\frac{1}{n} \mat S_n$.  Thus, by Lemma \ref{lemma:eigcriterion}, a.s., for $n$ sufficiently large, $z \in \tilde{\mathcal{T}}_n(\kappa, \delta', \eps)$ is an eigenvalue of $\frac{1}{n} \mat{S}_n (\mat I + \mat P_n)$ if and only if $f(z) = 0$.  Moreover, by Theorem \ref{thm:sc:isotropic}, we obtain (as before) that, a.s., 
$$ f(z) = g(z) + o(1) $$
uniformly for $z \in \tilde{\mathcal{T}}_n(\kappa, \delta', \eps)$.  

In order to reach a contradiction, suppose that $z \in \tilde{\mathcal{T}}_n(\kappa, \delta', \eps)$ is an eigenvalue of $\frac{1}{n} \mat{S}_n (\mat I + \mat P_n)$.  Then, a.s.
$$ g(z) = o(1). $$
By definition of $g$, this implies that there exists $1 \leq i \leq k$ such that
$$ 1 + (1 + z m_{\mathrm{MP}}(z)) \lambda_i(\mat P_n) = o(1). $$
By Lemma \ref{lemma:mMP} and condition \eqref{eq:sc:noeigdelta}, it must be the case that $\lambda_i(\mat P_n)$ satisfies \eqref{eq:sc:largeevalues}.  In other words, $1 \leq i \leq j$.  Hence, we obtain that
$$ z m_{\mathrm{MP}}(z) = -1 - \frac{1}{\lambda_i(\mat P_n)} + o(1). $$
In view of \eqref{eq:mMP} and the fact that $\|P_n\| = O(1)$, we find that 
$$ z = \lambda_i(\mat P_n) \left( 1 + \frac{1}{\lambda_i(\mat P_n)} \right)^2 + o(1). $$
This contradicts \eqref{eq:sc:3delta} and the assumption that $z \in \tilde{\mathcal{T}}_n(\kappa, \delta',\eps)$.  Therefore, we conclude that a.s., for $n$ sufficiently large, no eigenvalue of $\frac{1}{n} \mat S_n (\mat I + \mat P_n)$ is contained in $\tilde{\mathcal{T}}_n(\kappa, \delta',\eps)$.
\end{proof}

It remains to verify the following.  

\begin{lemma} \label{lem:sc:ident}
Let $\lambda \in \mathbb{C}$ with $||\lambda| - 1| \geq \delta$, for some $\delta > 0$.  If $1 + (1 + z m_{\mathrm{MP}}(z)) \lambda = 0$, then $|\lambda| \geq 1+\delta$ and $z = \lambda \left( 1 + \frac{1}{\lambda} \right)^2$.  
\end{lemma}
\begin{proof}
Suppose $1 + (1 + z m_{\mathrm{MP}}(z)) \lambda = 0$.  From Lemma \ref{lemma:mMP}, it must be the case that $|\lambda| \geq 1$.  Thus, by assumption, $|\lambda| \geq 1 + \delta$.  Hence, we obtain that 
\begin{equation} \label{eq:idzmMP}
	z m_{\mathrm{MP}}(z) = -1 - \frac{1}{\lambda}. 
\end{equation}
Substituting \eqref{eq:idzmMP} into \eqref{eq:mMP} yields the solution
$$ z = \lambda \left(1 + \frac{1}{\lambda} \right)^2, $$  
and the proof is complete.  
\end{proof}

\section{Proof of results in Section \ref{sec:eigenvectors}} \label{sec:proof:eigenvectors}

In order to prove Theorem \ref{thm:eigenvectors} and Theorem \ref{thm:SC:eigenvectors}, we will need the following concentration result.

\begin{lemma} \label{lemma:conc}
For each $n \geq 1$, let $\mat{A}_n$ be a deterministic $n \times n$ matrix and let $u_n$ be a random vector uniformly distributed on the unit sphere in $\mathbb{R}^n$ or $\mathbb{C}^n$.  If there exists a constant $C > 0$ such that $\sup_{n} \|\mat{A}_n\| \leq C$, then almost surely, $$ \lim_{n \to \infty} \left| u_n^\ast \mat{A}_n u_n - \frac{1}{n} \tr \mat{A}_n \right| = 0. $$
\end{lemma}
\begin{proof}
Since $u_n$ is a random vector uniformly distributed on the unit sphere in $\mathbb{R}^n$ (or, respectively, $\mathbb{C}^n$), $u_n$ has the same distribution as $q_n / \|q_n\|$, where 
$$ q_n := (\xi_1, \ldots, \xi_n) $$
and $\xi_1 \ldots, \xi_n$ are iid standard real (respectively, complex) Gaussian random variables.  Thus, it suffices to show that almost surely
$$ \frac{q_n^\ast}{\|q_n\|} \mat{A}_n \frac{q_n}{\|q_n\|} - \frac{1}{n} \tr \mat{A}_n \longrightarrow 0 $$
as $n \to \infty$.  

By the Hanson--Wright inequality (see\footnote{Most authors choose to state the Hanson--Wright inequality for the case when both the random vector and deterministic matrix are real.  The result easily extends to the complex setting; see \cite[Section 3.1]{RVhw} for a standard complexification trick.}, for instance, \cite[Theorem 1.1]{RVhw} or \cite[Theorem 2.3]{A}), we obtain that, for any $t > 0$, 
$$ \Prob \left( \frac{1}{n} \left| q_n^\ast \mat{A}_n q_n - \tr \mat{A}_n \right| \geq t \right) \leq C' \exp(- c' n \min\{t^2,t\}) $$
for some constants $C', c' > 0$ depending only on $C$.  Here, we used the bound $\|\mat{A}_n\|_2^2 \leq n \|\mat{A}_n\|^2 \leq Cn$.  Thus, by the Borel--Cantelli lemma, we conclude that
\begin{equation} \label{eq:asconc}
	\frac{1}{n} \left| q_n^\ast \mat{A}_n q_n - \tr \mat{A}_n \right| \longrightarrow 0 
\end{equation}
almost surely as $n \to \infty$.  

Therefore, by the triangle inequality, we have
\begin{align*}
	\left| \frac{1}{\|q_n\|^2} q_n^\ast \mat{A}_n q_n - \frac{1}{n} \tr \mat{A}_n \right| &\leq \left| \frac{q_n^\ast \mat{A}_n q_n}{\|q_n\|^2} - \frac{q_n^\ast \mat{A}_n q_n}{n} \right| + \frac{1}{n} \left| q_n^\ast \mat{A}_n q_n - \tr \mat{A}_n \right| \\
	&\leq \|\mat{A}_n\| \|q_n\|^2 \left| \frac{1}{\|q_n\|^2} - \frac{1}{n} \right| + \frac{1}{n} \left| q_n^\ast \mat{A}_n q_n - \tr \mat{A}_n \right| \\
	&\leq C \left| 1- \frac{\|q_n\|^2}{n} \right| + \frac{1}{n} \left| q_n^\ast \mat{A}_n q_n - \tr \mat{A}_n \right|.
\end{align*}
The first term on the right-hand side converges almost surely to zero by the law of large numbers.  The second term converges to zero by \eqref{eq:asconc}.
\end{proof}

We also require the following well-known bound from \cite{BSbook} for the spectral norm of a Wigner random matrix.  

\begin{theorem}[Spectral norm of a Wigner matrix; Theorem 5.1 from \cite{BSbook}] \label{thm:wigner:norm}
Let $\xi$ and $\zeta$ be real random variables.  Assume $\xi$ has mean zero, unit variance, and finite fourth moment;  suppose $\zeta$ has mean zero and finite variance.  For each $n \geq 1$, let $\mat{W}_n$ be an $n \times n$ Wigner matrix with atom variables $\xi$ and $\zeta$.  Then a.s.
$$ \limsup_{n \to \infty} \frac{1}{\sqrt{n}} \| \mat W_n \| \leq 2. $$
\end{theorem}

We now prove Theorem \ref{thm:eigenvectors} and Theorem \ref{thm:SC:eigenvectors}.  The proofs presented here are based on the arguments given in \cite{BGN}.    

\begin{proof}[Proof of Theorem \ref{thm:eigenvectors}]
For convenience, let $\lambda_1, \ldots, \lambda_n$ denote the eigenvalues of $\frac{1}{\sqrt{n}} \mat W_n + \mat P_n$.  Recall that $\tilde{\theta} := \theta + 1/\theta$ and $|\theta| > 1$.  By Lemma \ref{lem:ellipse}, there exists $\delta' > 0$ such that $\delta'< \frac{\delta^2}{1+\delta}$ and $\tilde{\theta} \not\in \Esc_{3\delta'}$.  By Theorem \ref{thm:or}, a.s., for $n$ sufficiently large, only one eigenvalue of $\frac{1}{\sqrt{n}} \mat W_n + \mat P_n$ is not in $\Esc_{2\delta'}$ and, after labeling the eigenvalues correctly, this eigenvalue a.s.~satisfies
\begin{equation} \label{eq:lambdanconv}
	\lambda_n = \tilde{\theta} + o(1). 
\end{equation}

It follows from Theorem \ref{thm:wigner:norm}, that a.s., for $n$ sufficiently large, 
\begin{equation} \label{eq:normbnd}
	\left\| \frac{1}{\sqrt{n}} \mat W_n \right\| \leq 2 + \delta' 
\end{equation}
Thus, a.s., for $n$ sufficiently large, $\frac{1}{\sqrt{n}} \mat W_n - \lambda_n \mat I$ is invertible.  In addition, since $\tilde{\theta} \not\in \Esc_{3\delta'}$, there exists $\eps > 0$ such that, a.s.
\begin{equation} \label{eq:epsbnds1}
	\left\| \left( \frac{1}{\sqrt{n}} \mat W_n - \tilde{\theta} \mat I \right)^{-1} \right\| \leq \frac{1}{\eps} \quad \text{and} \quad \left\| \left( \frac{1}{\sqrt{n}} \mat W_n - \lambda_n \mat I \right)^{-1} \right\| \leq \frac{1}{\eps} 
\end{equation}
for $n$ sufficiently large (see, e.g., \cite[Theorem~2.2]{TE}).  By symmetry, we also conclude that a.s.~for $n$ sufficiently large
\begin{equation} \label{eq:epsbnds2}
	\left\| \left( \frac{1}{\sqrt{n}} \mat W_n - \overline{\tilde{\theta}} \mat I \right)^{-1} \right\| \le \frac1\eps\quad \text{and} \quad \left\| \left( \frac{1}{\sqrt{n}} \mat W_n - \overline{\lambda_n} \mat I \right)^{-1} \right\| \leq \frac{1}{\eps}. 
\end{equation}

Let $v_n$ be a unit eigenvector of $\frac{1}{\sqrt{n}} \mat{W}_n + \mat{P}_n$ with corresponding eigenvalue $\lambda_n$.  Since 
$$ \left( \frac{1}{\sqrt{n}} \mat{W}_n + \mat{P}_n \right) v_n = \lambda_n v_n, $$
we have
$$ \left( \frac{1}{\sqrt{n}} \mat{W}_n - \lambda_n \mat I \right) v_n = - \mat{P}_n v_n = (-\theta u_n^\ast v_n) u_n. $$
Hence, a.s, for $n$ sufficiently large, we obtain
$$ v_n = (- \theta u_n^\ast v_n) \mat{G}_n(\lambda_n) u_n, $$
where 
$$ \mat G_n(z) := \left( \frac{1}{\sqrt{n}} \mat{W}_n - z \mat I \right)^{-1} $$
is the resolvent of $\frac{1}{\sqrt{n}} \mat W_n$.  As $v_n$ is a unit vector, this implies that
$$ v_n = \frac{ - \theta u_n^\ast v_n}{|\theta u_n^\ast v_n|} \frac{ \mat{G}_n(\lambda_n) u_n}{\sqrt{ u_n^\ast \mat{G}_n(\overline{\lambda_n}) \mat{G}_n(\lambda_n) u_n}}. $$
Therefore, we find that 
\begin{equation} \label{eq:eigenvectordef}
	|u_n^\ast v_n|^2 = \frac{ \left| u_n^\ast \mat{G}_n(\lambda_n) u_n \right|^2}{ u_n^\ast \mat{G}_n(\overline{\lambda_n}) \mat{G}_n(\lambda_n)  u_n }. 
\end{equation}

By the resolvent identity,
\begin{align*}
	&\left| u_n^\ast \left( \frac{1}{\sqrt{n}} \mat W_n - \lambda_n \mat I \right)^{-1}u_n  - u_n^\ast \left( \frac{1}{\sqrt{n}} \mat W_n - \tilde{\theta} \mat I \right)^{-1}u_n \right| \\
	&\qquad\qquad\qquad\qquad\leq |\lambda_n - \tilde{\theta}|  \left\|  \left( \frac{1}{\sqrt{n}} \mat W_n - \lambda_n \mat I \right)^{-1} \right\| \left\| \left( \frac{1}{\sqrt{n}} \mat W_n - \tilde{\theta} \mat I \right)^{-1}  \right\|
\end{align*}
using the fact that that $\lambda_n$ and $\tilde{\theta}$ are not eigenvalues of $\frac{1}{\sqrt{n}} \mat W_n$.   Combining this inequality with \eqref{eq:lambdanconv} and \eqref{eq:epsbnds1}, we obtain that, a.s., 
$$ \left| u_n^\ast \mat G_n(\lambda_n) u_n - u_n^\ast \mat G_n(\tilde{\theta}) u_n \right| \leq \frac{|\lambda_n - \tilde{\theta}|}{\eps^2} = o(1). $$
Thus, since $\mat W_n$ and $u_n$ are independent, Lemma \ref{lemma:conc} implies that a.s.
$$ u_n^\ast \mat G_n(\lambda_n) u_n - \frac{1}{n} \tr \mat G_n(\tilde{\theta}) \longrightarrow 0 $$
as $n \to \infty$.  Let $f: \mathbb{R} \to \mathbb{C}$ be a bounded continuous function such that
$$ f(x) := \frac{1}{x - \hat{\theta}} \quad \text{for } x \in \mathbb{R} \cap \Esc_{2\delta'} $$
and $f(x) := 0$ for $x \in \mathbb{R} \setminus \Esc_{3\delta'}$.  In view of \eqref{eq:normbnd}, it follows that a.s., for $n$ sufficiently large, 
$$ \frac{1}{n} \tr \mat G_n(\tilde{\theta}) = \frac{1}{n} \sum_{i=1}^n f \left(\frac{1}{\sqrt{n}} \lambda_i(\mat W_n) \right). $$
Thus, by the semicircle law (Theorem \ref{thm:wigner}), we have that a.s.
$$ \frac{1}{n} \tr \mat G_n(\tilde{\theta}) \longrightarrow \int_{-\infty}^\infty \frac{\rho_{\mathrm{sc}}(x)}{x - \tilde{\theta}} dx $$
as $n \to \infty$.  Hence, we conclude that a.s.
\begin{equation} \label{eq:numGn}
	u_n^\ast \mat G_n(\lambda_n) u_n \longrightarrow \int_{-\infty}^\infty \frac{\rho_{\mathrm{sc}}(x)}{x - \tilde{\theta}} dx 
\end{equation}
as $n \to \infty$.  Using the bounds in \eqref{eq:epsbnds1} and \eqref{eq:epsbnds2}, a similar argument reveals that a.s.
\begin{equation} \label{eq:denGn}
	u_n^\ast \mat {G}_n(\overline{\lambda_n}) \mat G_n(\lambda_n) u_n \longrightarrow \int_{-\infty}^\infty \frac{\rho_{\mathrm{sc}}(x)}{|x - \tilde{\theta}|^2}dx 
\end{equation}
as $n \to \infty$.  Combining \eqref{eq:numGn} and \eqref{eq:denGn} with \eqref{eq:eigenvectordef} completes the proof of the theorem.  
\end{proof}

We now turn to the proof of Theorem~\ref{thm:SC:eigenvectors}.

\begin{proof}[Proof of Theorem~\ref{thm:SC:eigenvectors}]
Let $\lambda_1,\dots , \lambda_n$ be the eigenvalues of $\frac1n \mat S_n(\mat I+ \mat P_n)$.  Recall that $\mat P_n = \theta u_n {u_n}^\ast$ where $\theta \in \mathbb C$ and $\abs{\theta} >1$.  By Theorem~\ref{thm:sc:refine}, a.s.~for $n$ sufficiently large, only one eigenvalue of $\frac1n \mat S_n(\mat I+ \mat P_n)$ lies outside of $\Emp_{2\delta}$ for some $\delta > 0$, and this eigenvalue satisfies 
\begin{equation}\label{eqn:sc:lambdanconv}
\lambda_n=\hat\theta+o(1)
\end{equation}
after renumbering properly, where $\hat\theta := \theta\left(1+\frac1\theta\right)^2$.  Moreover, the remaining eigenvalues a.s.~lie inside $\Emp_{\delta}$.  

Note that by Lemma~\ref{lemma:sc:largest} (and the fact that the eigenvalues of $\frac{1}{n} \mat S_n$ are real and non-negative), the matrix $\frac1n \mat S_n -\lambda_n \mat I$ is a.s.~invertible for sufficiently large $n$.  Furthermore, the norm of the inverse of $\frac1n \mat S_n -\lambda_n \mat I$ is well behaved; in particular, since $\hat\theta \notin \Emp_{\delta}$, there exists $\eps >0$ such that 
\begin{equation}\label{eqn:sc:epsbnds1}
\left\| \left(\frac1n \mat S_n -\lambda_n \mat I\right)^{-1} \right\| < \frac1\eps 
\qquad \mbox{and}\qquad
\left\| \left(\frac1n \mat S_n -\hat\theta \mat I\right)^{-1} \right\| < \frac1\eps 
\end{equation}
a.s.~for $n$ sufficiently large (see, e.g., \cite[Theorem~2.2]{TE}).  By symmetry, the above bounds imply
\begin{equation}\label{eqn:sc:epsbnds2}
\left\| \left(\frac1n \mat S_n -\overline{\lambda_n} \mat I\right)^{-1} \right\| < \frac1\eps 
\qquad \mbox{and}\qquad
\left\| \left(\frac1n \mat S_n -\overline{\hat\theta} \mat I\right)^{-1} \right\| < \frac1\eps 
\end{equation}
a.s.~for $n$ sufficiently large.

Let $v_n$ be a unit eigenvector of $\frac1n \mat S_n(\mat I+ \mat P_n)$ with eigenvalue $\lambda_n$, so 
$$\frac1n \mat S_n(\mat I+ \mat P_n)v_n=\lambda_n v_n.$$
Rearranging this equality, we have
\begin{align*}
v_n &= -(\theta u_n^\ast v_n) \mat R_n(\lambda_n)\frac1n\mat S_n u_n 
\end{align*}
where 
$$ \mat R_n(z) := \left(\frac1n \mat S_n - z \mat I\right)^{-1} $$  
is the resolvent of $\frac{1}{n} \mat S_n$.  Because $v_n$ is a unit vector, we arrive at 
$$v_n= \frac{ -\theta u_n^\ast v_n }{ | \theta u_n^\ast v_n| } \frac{ \mat R_n(\lambda_n)\frac1n\mat S_n u_n}{\sqrt{u_n^\ast \frac1n S_n\mat R_n(\bar{\lambda}_n) \mat R_n(\lambda_n)\frac1n\mat S_n u_n}}.$$

Thus, we obtain
\begin{equation}\label{eq:sc:eigenvectordef}
|u_n^\ast v_n|^2 = 
\frac{| u_n^\ast \mat R_n(\lambda_n)\frac1n\mat S_n u_n |^2}
{ u_n^\ast \frac1n S_n\mat R_n(\bar{\lambda}_n) \mat R_n(\lambda_n)\frac1n\mat S_n u_n},
\end{equation}
and we will now proceed to show that the numerator and denominator of the above tend to the desired limits.

For the numerator, using Lemma~\ref{lemma:sc:largest}, the resolvent identity, \eqref{eqn:sc:lambdanconv}, and \eqref{eqn:sc:epsbnds1}, we note that a.s.
\begin{align*}
	\left| u_n^\ast \mat R_n(\lambda_n)\frac1n\mat S_n u_n - u_n^\ast \mat R_n(\hat\theta)\frac1n\mat S_n u_n\right| &\le \|\mat R_n(\lambda_n) -\mat R_n(\hat\theta)   \| \left\|\frac1n \mat S_n\right\| \\
	&\le 5 |\lambda_n-\hat\theta|\pfrac{1}{\eps^2} \\
	&= o(1).
\end{align*}

Using the fact that $u_n$ is independent of $\mat R_n(\hat\theta)\frac1n\mat S_n$, we can apply Lemma~\ref{lemma:conc} to show that
$$ u_n^\ast \mat R_n(\lambda_n)\frac1n\mat S_n u_n- \frac1n \tr\left(\mat R_n(\hat\theta)\frac1n\mat S_n\right) \longrightarrow 0 $$
a.s.~as $n\to\infty$.  Let $f: \mathbb{R} \to \mathbb{C}$ be a bounded continuous function such that
$$ f(x) := \frac{x}{x - \hat\theta} \quad \text{for } x \in \mathbb{R} \cap \Emp_{\delta} $$
and $f(x) := 0$ for $x \in \mathbb{R} \setminus \Emp_{2\delta}$.  Notice that
$$
\frac1n \tr\left(\mat R_n(\hat\theta)\frac1n\mat S_n\right) = \frac1n \sum_{i=1}^n \frac{\frac1n\lambda_i(\mat S_n)}{\frac1n\lambda_i(\mat S_n) + \hat\theta} = \frac1n\sum_{i=1}^n f\left(\frac1n \lambda_i(\mat S_n) \right).
$$

Thus, using Theorem~\ref{thm:mp}, we have that a.s.
\begin{equation}\label{numSC}
u_n^\ast \mat R_n(\lambda_n)\frac1n\mat S_n u_n
\longrightarrow
\int \frac{x\rho_{\mathrm{MP},1}(x)}{x-\hat\theta}\,dx
\end{equation}
as $n\to\infty$.  Using a similar argument along with the bounds in \eqref{eqn:sc:epsbnds1} and \eqref{eqn:sc:epsbnds2}, one can show that a.s.
\begin{align}
u_n^\ast \frac1n \mat S_n\mat R_n(\bar{\lambda}_n) \mat R_n(\lambda_n)\frac1n\mat S_n u_n
&=
\frac1n\tr\left(\frac1n \mat  S_n\mat R_n(\bar{\lambda}_n) \mat R_n(\lambda_n)\frac1n\mat S_n\right) +o(1) \nonumber \\
&\longrightarrow
\int \frac{x^2\rho_{\mathrm{MP},1}(x)}{\left|x-\hat\theta\right|^2}\,dx. \label{denomSC}
\end{align}

Combining \eqref{numSC} and \eqref{denomSC} with \eqref{eq:sc:eigenvectordef} completes the proof.
\end{proof}

\section{Proof of Theorem \ref{thm:wigner:critical}} \label{sec:proof:critical}

In order to prove Theorem \ref{thm:wigner:critical}, we will need the Tao--Vu replacement principle from \cite{TVesd}, which we present below.  Note that we use $\log$ to denote the natural logarithm.

\begin{theorem}[Replacement principle; Theorem 2.1 from \cite{TVesd}] \label{thm:replacement}
For each $n \geq 1$, let $\mat A_n$ and $\mat B_n$ be random $n \times n$ matrices with complex-valued entries.  Assume that
\begin{enumerate}[(i)]
\item the expression
$$ \frac{1}{n^2} \| \mat A_n \|^2_2 + \frac{1}{n^2} \| \mat B_n \|^2_2 $$
is bounded in probability (resp. almost surely), and
\item
for a.e. $z \in \mathbb{C}$, 
$$ \frac{1}{n} \log \left| \det \left( \frac{1}{\sqrt{n}} \mat A_n - z \mat I \right) \right| - \frac{1}{n} \log \left| \det \left( \frac{1}{\sqrt{n}} \mat B_n - z \mat I \right) \right| $$
converges in probability (resp. almost surely) to zero as $n \to \infty$.  In particular, for each fixed $z$, these determinants are nonzero with probability $1 - o(1)$ (resp. almost surely for $n$ sufficiently large).  
\end{enumerate}
Then, for any bounded continuous function $f: \mathbb{C} \to \mathbb{C}$, 
$$ \int f d\mu_{\frac{1}{\sqrt{n}} \mat A_n} - \int f d \mu_{\frac{1}{\sqrt{n}} \mat B_n} $$
converges in probability (resp. almost surely) to zero as $n \to \infty$.  
\end{theorem}

We will also need the following result from \cite{CN}. 

\begin{lemma}[Companion matrix for the critical points] \label{lemma:companion}
If $x_1, \ldots, x_n \in \mathbb{C}$ are the roots of $p(z) := \prod_{j=1}^n (z - x_j)$, and $p$ has critical points $y_1, \ldots, y_{n-1}$ counted with multiplicity, then the matrix
$$ \mat D_{n} \left(\mat I_{n} - \frac{1}{n} \mat J_{n} \right) $$
has $y_1, \ldots, y_{n-1}, 0$ as its eigenvalues, where $\mat D_{n} := \diag(x_1, \ldots,x_{n})$, $\mat I_{n}$ is the identity matrix of order $n$, and $\mat J_{n}$ is the $n \times n$ matrix of all entries $1$.  
\end{lemma}

In order to prove Theorem \ref{thm:wigner:critical}, we will use Theorem \ref{thm:replacement} to compare the eigenvalues of $\frac{1}{\sqrt{n}} \mat W_n + \mat P_n$ with the eigenvalues of the companion matrix defined in Lemma \ref{lemma:companion}.  Since Theorem \ref{thm:global:sc} describes the limiting behavior of the empirical spectral measure of $\frac{1}{\sqrt{n}} \mat W_n + \mat P_n$, we would then conclude that $\mu'_{\frac{1}{\sqrt{n}} \mat W_n + \mat P_n}$ also converges to the same limit.  (Note that the companion matrix defined in Lemma \ref{lemma:companion} contains an extra eigenvalue at the origin, but this single eigenvalue does not effect the limiting empirical spectral measure.)  We require the following lemma before we present the proof of Theorem \ref{thm:wigner:critical}.  


\begin{lemma} \label{lemma:lowbndst}
Assume $\xi$, $\zeta$, $\mat W_n$, and $\mat P_n$ satisfy the assumptions of Theorem \ref{thm:wigner:refine}.  In addition, for $n$ sufficiently large, assume that the $j$ eigenvalues of $\mat P_n$ which satisfy \eqref{eq:largeevalues} do not depend on $n$.  Then, for a.e. $z \in \mathbb{C}$, there exists $c > 0$ (depending on $z$) such that a.s., for $n$ sufficiently large,
\begin{equation} \label{eq:lambdailowerbnd}
	\inf_{1 \leq i \leq n} \left| \lambda_i \left( \frac{1}{\sqrt{n}} \mat W_n + \mat P_n \right) - z \right| \geq c. 
\end{equation}
In addition, for a.e. $z \in \mathbb{C}$, a.s.
\begin{equation} \label{eq:lambdaisum}
	\lim_{n \to \infty} \frac{1}{n} \sum_{i=1}^n \frac{1}{ \lambda_i \left( \frac{1}{\sqrt{n}} \mat W_n + \mat P_n \right) - z } \longrightarrow m_{\mathrm{sc}}(z), 
\end{equation}
where $m_{\mathrm{sc}}$ is the Stieltjes transform of the semicircle law defined in \eqref{eq:mscdef}.  
\end{lemma}
\begin{proof}
We take $n$ sufficiently large so that $\lambda_1(\mat P_n), \ldots, \lambda_j(\mat P_n)$ are the $j$ eigenvalues of $\mat P_n$ which satisfy \eqref{eq:largeevalues} and which do not depend on $n$.  Define the set $Q \subset \mathbb{C}$ by
$$ Q := \mathbb{R} \bigcup \left\{ \lambda_i(\mat P_n) + \frac{1}{\lambda_i(\mat P_n)} : 1 \leq i \leq j \right\}. $$
Clearly, $Q \subset \mathbb{C}$ has (Lebesgue) measure zero, and, by construction, $Q$ does not depend on $n$.  

Fix $z \in \mathbb{C}$ with $z \notin Q$.  Set $c := \dist(z,Q)/2$.  Recall that $\Lambda\left( \frac{1}{\sqrt{n}} \mat W_n + \mat P_n \right)$ denotes the set of eigenvalues of $\frac{1}{\sqrt{n}} \mat W_n + \mat P_n$.  By Theorem \ref{thm:wigner:refine}, it follows that, a.s., for $n$ sufficiently large, $\Lambda\left( \frac{1}{\sqrt{n}} \mat W_n + \mat P_n \right) \subset \EQ_c,$ where $\EQ_c := \{w \in \mathbb C : \dist(w,Q)\le c\}$.  Also, by the definition of $c$, we have $\dist(z, \EQ_c) \ge c$.  Thus, we have 
$$\dist\left(\Lambda\left( \frac{1}{\sqrt{n}} \mat W_n + \mat P_n \right), z \right) \ge \dist(\EQ_c, z) \ge c$$ almost surely for $n$ sufficiently large.
%
This verifies \eqref{eq:lambdailowerbnd}.

For \eqref{eq:lambdaisum}, we again fix $z \notin Q$.  Let $c$ be defined as above.  Let $f: \mathbb{C} \to \mathbb{C}$ be a bounded continuous function such that $f(\lambda) :=  \frac{1}{\lambda - z}$ for $|\lambda - z| \geq c$ and $f(\lambda) := 0$ for $|\lambda - z| < c/2$.  By \eqref{eq:lambdailowerbnd}, it follows that, a.s., for $n$ sufficiently large, 
$$ \frac{1}{n} \sum_{i=1}^n \frac{1}{ \lambda_i \left( \frac{1}{\sqrt{n}} \mat W_n + \mat P_n \right) - z } = \frac{1}{n} \sum_{i=1}^n f \left( \lambda_i \left( \frac{1}{\sqrt{n}} \mat W_n + \mat P_n \right) \right). $$
Thus, by Theorem \ref{thm:global:sc}, we conclude that a.s. 
$$ \frac{1}{n} \sum_{i=1}^n f \left( \lambda_i \left( \frac{1}{\sqrt{n}} \mat W_n + \mat P_n \right) \right) \longrightarrow \int_{\mathbb{R}} \frac{\rho_{\mathrm{sc}}(x)}{x - z}dx = m_{\mathrm{sc}}(z) $$
as $n \to \infty$, and the proof is complete.  
\end{proof}

We now prove Theorem \ref{thm:wigner:critical}.

\begin{proof}[Proof of Theorem \ref{thm:wigner:critical}]
Let $\lambda_1, \ldots, \lambda_n$ be the eigenvalues of $\mat{W}_n + \sqrt{n} \mat P_n$.  Let $\mat D_{n} := \diag(\lambda_1, \ldots, \lambda_{n})$, and let $\mat J_{n}$ denote the $n \times n$ matrix of all ones.  In view of Lemma \ref{lemma:companion}, it follows that, for any bounded and continuous function $f: \mathbb{C} \to \mathbb{C}$, 
$$ \int f d \mu'_{\frac{1}{\sqrt{n}} \mat W_n + \mat P_n} = \int f d \mu_{\frac{1}{\sqrt{n}} \mat D_n \left( \mat I_n - \frac{1}{n} \mat J_n \right)} + o(1). $$
(The $o(1)$ term is due to the fact that the matrix $\frac{1}{\sqrt{n}} \mat D_n \left( \mat I_n - \frac{1}{n} \mat J_n \right)$ has a deterministic eigenvalue at the origin, which is not a critical point.)  In addition, Theorem \ref{thm:global:sc} implies that a.s.
$$ \int f d \mu_{\frac{1}{\sqrt{n}} \mat D_n} \longrightarrow \int f d \mu_{\mathrm{sc}} $$
as $n \to \infty$.  Therefore, it suffices to show that a.s.
\begin{equation} \label{eq:showfdmu}
	\int f d \mu_{\frac{1}{\sqrt{n}} \mat D_n \left( \mat I_n - \frac{1}{n} \mat J_n \right)} - \int f d \mu_{\frac{1}{\sqrt{n}} \mat D_n} \longrightarrow 0 
\end{equation}
as $n \to \infty$.  We will use Theorem \ref{thm:replacement} to verify \eqref{eq:showfdmu}.  

We begin by observing that 
\begin{align*}
	\frac{1}{n^2} \| \mat D_{n} \|_2^2 &= \frac{1}{n^2} \sum_{i=1}^{n} |\lambda_i|^2 \\
		&\leq \frac{1}{n} \| \mat W_n + \sqrt{n} \mat P_n \|^2 \\
		&\leq \frac{2}{n} \| \mat W_n \|^2 + 2 \| \mat P_n \|^2,
\end{align*}
where the first inequality above uses the fact that the spectral radius of any matrix is bounded above by its spectral norm. Hence, by Theorem \ref{thm:wigner:norm} and the assumption that $\|\mat P_n\| = O(1)$, it follows that $\frac{1}{n^2} \| \mat D_{n} \|_2^2$ is a.s.~bounded.  Similarly, we compute
\begin{align*}
\frac{1}{n} &\left \| \mat D_{n} \left( \mat I_{n} - \frac{1}{n} \mat J_{n} \right)  \right \|_2 
\leq  
\frac{1}{n} \norm{\mat D_{n}}_2 +
\frac{1}{n} \norm{\mat D_{n}}_2 \norm{\frac{1}{n} \mat J_{n}}_2,
\end{align*}
which shows that the left-hand side above is a.s.~bounded by noting that $\frac{1}{n} \norm{\mat D_{n}}_2$ is a.s.~bounded from before and $ \norm{\frac{1}{n}\mat J_{n}}_2=  1$.  


In view of Theorem \ref{thm:replacement}, it remains to show that, for a.e. $z \in \mathbb{C}$, a.s.
\begin{align*}
	\frac{1}{n} &\log \left| \det \left[ \frac{1}{\sqrt{n}} \mat D_{n} \left( \mat I_{n} - \frac{1}{n} \mat J_{n} \right)  - z \mat I_{n} \right] \right| 
	- \frac{1}{n} \log \left| \det \left[ \frac{1}{\sqrt{n}} \mat D_{n} - z \mat I_{n} \right] \right| \longrightarrow 0 
\end{align*}
as $n \to \infty$.  

Define $\tilde{\mat D}_{n} := \frac{1}{\sqrt{n}} \mat D_{n}$, and set $\tilde{\lambda}_i := \frac{1}{\sqrt{n}} \lambda_i$ for $1 \leq i \leq n$.  By Lemma \ref{lemma:lowbndst}, we fix $z \in \mathbb{C} \setminus \mathbb{R}$ such that a.s., for $n$ sufficiently large
\begin{equation} \label{eq:lambdailowbnd}
	\inf_{1 \leq i \leq n} | \tilde{\lambda}_i - z| \geq c
\end{equation}
for some $c > 0$ (depending on $z$), and a.s.
\begin{equation} \label{eq:lambdaiconv}
	\frac{1}{n} \sum_{i=1}^n \frac{1}{\tilde{\lambda}_i - z } \longrightarrow m_{\mathrm{sc}}(z) 
\end{equation}
as $n \to \infty$.  

We first observe that, by \eqref{eq:lambdailowbnd} and \eqref{eq:lambdaiconv}, we have a.s.
\begin{align*}
	1 - \frac{1}{n} \sum_{i=1}^{n} \frac{ \tilde{\lambda}_i}{\tilde{\lambda}_i - z } &= -\frac{z}{n} \sum_{i=1}^n \frac{1}{\tilde{\lambda}_i - z} = - zm_{\mathrm{sc}}(z) + o(1). 
\end{align*}
In particular, since $z \not\in \mathbb{R}$, we obtain that a.s., for $n$ sufficiently large, 
$$ \left| 1 - \frac{1}{n} \sum_{i=1}^{n} \frac{ \tilde{\lambda}_i}{\tilde{\lambda}_i - z } \right| $$
is bounded away from zero and is also bounded above by a constant (depending on $z$).  Hence, 
\begin{equation} \label{eq:zerolimln}
	\lim_{n \to \infty} \frac{1}{n} \log \left| 1 - \frac{1}{n} \sum_{i=1}^{n} \frac{ \tilde{\lambda}_i}{\tilde{\lambda}_i - z }   \right| = 0 
\end{equation}
a.s.  However, we now notice that
\begin{align*}
	1 - \frac{1}{n} \sum_{i=1}^{n} \frac{ \tilde{\lambda}_i}{\tilde{\lambda}_i - z } &= 1 - \frac{1}{n} \tr \left[ \left( \tilde{\mat D}_n - z \mat I_n \right)^{-1} \tilde{\mat D}_n \right] \\
	&= 1 - \frac{1}{n} u_n^\mathrm{T} \left( \tilde{\mat D}_n - z \mat I_n \right)^{-1} \tilde{\mat D}_n u_n, 
\end{align*}
where $u_{n}$ is the $n$-dimensional vector of all ones.  Since $u_{n} u_{n}^\mathrm{T} = \mat J_{n}$, we apply \eqref{eq:fundamental} to obtain
\begin{align*}
	1 - \frac{1}{n} \sum_{i=1}^{n} \frac{ \tilde{\lambda}_i}{\tilde{\lambda}_i - z } &= \det \left[ \mat I_{n} - \left( \tilde{\mat D}_n - z \mat I_n \right)^{-1} \tilde{\mat D}_n \frac{1}{n} \mat J_n \right] \\
	&= \det \left[ \tilde{\mat D}_n \left( \mat I_n - \frac{1}{n} \mat J_n \right) - z \mat I_n\right] \det \left[\left( \tilde{\mat D}_n - z \mat I_n \right)^{-1} \right]. 
\end{align*}
Therefore, by \eqref{eq:lambdailowbnd}, a.s., for $n$ sufficiently large, we conclude that the determinant 
$$ \det \left( \tilde{\mat D}_n \left( \mat I_n - \frac{1}{n} \mat J_n \right) - z\mat I_n \right) $$ 
is nonzero.  In addition, by taking logarithms of both sides and applying \eqref{eq:zerolimln}, we conclude that a.s.
\begin{align*}
	\frac{1}{n} &\log \left| \det \left[ \frac{1}{\sqrt{n}} \mat D_{n} \left( \mat I_{n} - \frac{1}{n} \mat J_{n} \right)  - z \mat I_{n} \right] \right| 
	- \frac{1}{n} \log \left| \det \left[ \frac{1}{\sqrt{n}} \mat D_{n} - z \mat I_{n} \right] \right| \longrightarrow 0 
\end{align*}
as $n \to \infty$, and the proof of the theorem is complete.  
\end{proof}

\section*{Acknowledgments} 

We would like to thank 
Kevin Costello,
Yen Do, 
David Renfrew, 
and
Benedek Valko
for helpful discussions on the project.  We would also like to thank Yan Fyodorov for bringing references \cite{FSphys1996,FSphys1997,FKphys1999,FSphys2003,SFTphys1999} to our attention.  
Lastly, we thank the anonymous referee for a careful reading of the manuscript and valuable suggestions.

\appendix

\section{Proof of Lemma \ref{lemma:interlace}} \label{app:interlace}

In order to prove Lemma \ref{lemma:interlace}, we will need the following deterministic lemma.  

\begin{lemma} \label{lemma:nonzero}
Let 
$$ \mat{W} = \begin{pmatrix} \mat W' & X \\ X^\ast & d \end{pmatrix} $$
be an $n \times n$ Hermitian matrix, where $\mat W'$ is the upper-left $(n-1) \times (n-1)$ minor of $\mat W$, $X \in \mathbb{C}^{n-1}$, and $d \in \mathbb{R}$.  If an eigenvalue of $\mat{W}$ is equal to an eigenvalue of $\mat W'$, then there exists a unit eigenvector $w$ of $\mat{W}'$ such that $X^\ast w = 0$.  
\end{lemma}
\begin{proof}
Assume $\lambda$ is an eigenvalue of $\mat W$ with unit eigenvector $\begin{pmatrix} v \\ q \end{pmatrix}$, where $v \in \mathbb{C}^{n-1}$ and $q \in \mathbb{C}$.  In addition, assume $\lambda$ is an eigenvalue of $\mat W'$ with unit eigenvector $u \in \mathbb{C}^{n-1}$.  From the eigenvalue equation 
$$ \mat W \begin{pmatrix} v \\ q \end{pmatrix} = \lambda \begin{pmatrix} v \\ q \end{pmatrix}, $$
we obtain
\begin{align}
	\mat W' v + q X &= \lambda v, \label{eq:eig1} \\
	X^\ast v + qd &= \lambda q. \label{eq:eig2}
\end{align}
Since $u^\ast \mat W' = \lambda u^\ast$, we multiply \eqref{eq:eig1} on the left by $u^\ast$ to obtain $q u^\ast X = 0$.  In other words, either $u^\ast X = 0$ or $q = 0$. 

If $u^\ast X = 0$, then the proof is complete (since $u$ is a unit eigenvector of $\mat W'$).  Assume $q = 0$.  Then $v$ is a unit vector, and equations \eqref{eq:eig1} and \eqref{eq:eig2} imply that
\begin{align*}
	\mat W' v &= \lambda v, \\
	X^\ast v &= 0.
\end{align*}
In other words, $v$ is a unit eigenvector of $\mat W'$ and $X^\ast v = 0$. 
\end{proof}

We now prove Lemma \ref{lemma:interlace}.  By Cauchy's interlacing theorem (see, for instance, \cite[Corollary III.1.5]{Bhatia}), it follows that the eigenvalues of $\mat W$ weakly interlace with the eigenvalues of $\mat W'$.  Thus, in order to prove Lemma \ref{lemma:interlace}, we will need to show that this weak interlacing is actually strict.  To that end, suppose a $\mat{W}$ and $\mat{W}'$ have a common eigenvalue.  Lemma \ref{lemma:nonzero} implies that there is a unit eigenvector $w$ of $\mat{W}'$ such that $X$ is orthogonal to $w$.  The crucial part is that $X$ and $\mat W'$ are independent.  Hence, $X$ and $w$ are independent.  Let $\JOmega$ be the event that $X$ is orthogonal to a unit eigenvector of $\mat{W}'$.  Using the result \cite[Proposition 2.2]{TVsimple} yields the following bound on the probability of the event $\JOmega$.

\begin{lemma}[Tao-Vu; Proposition 2.2 from \cite{TVsimple}] \label{lemma:tveigenvector}
Under the assumptions of Lemma \ref{lemma:interlace}, for every $\alpha > 0$, there exists $C>0$ (depending on $\mu$ and $\alpha$) such that $\Prob(\JOmega) \leq C n^{-\alpha}$.  
\end{lemma}

The proof of the first part of Lemma \ref{lemma:interlace} is now complete by applying Lemma \ref{lemma:tveigenvector}.  

We now consider the case when $\xi$ and $\zeta$ are absolutely continuous random variables.  In view of Lemma \ref{lemma:nonzero}, it suffices to show that, with probability $1$, $X$ is not orthogonal to any unit eigenvector of $\mat{W}'$.  In this case, it is well-known that $\mat{W}'$ has simple spectrum with probability $1$; see, for instance, \cite[Section 2.3]{DG}.  Condition on the matrix $\mat{W}'$ to have simple spectrum.  The unit eigenvectors of $\mat{W}'$ are now uniquely determined up to sign.  Let $v_1, \ldots, v_{n-1}$ be the unit eigenvectors of $\mat{W}'$.  Since $X$ is a $(n-1)$-vector whose entries are iid copies of $\xi$ and since $X$ and $\mat{W}'$ are independent, it follows that
$$ \Prob(X^* v_i = 0) = 0 $$
for all $1 \leq i \leq n-1$.  Therefore, by the union bound, we conclude that, with probability $1$, $X$ is not orthogonal to any unit eigenvector of $\mat{W}$.

\section{Proof of Lemma \ref{lemma:distinct}} \label{app:distinct}

Results similar to Lemma~\ref{lemma:distinct} can be proven in a variety of ways (see for example, \cite[Section~2.3]{DG}).  Below we give a proof using symmetric polynomials.

We begin by first proving the assertion that the sample covariance matrix  $\mat S_n = \transp{\mat X_n}\mat X_n$ almost surely has $r:=\min\{m,n\}$ distinct, strictly positive eigenvalues.  

Recall that $\mat X_n$ is an $m$ by $n$ real matrix.  In the case that $n>m$, set $k:=n-r = n - \min\{m,n\}$ and note that the matrix $\mat S_n = \transp{\mat X_n}\mat X_n$ has at least $k$ eigenvalues equal to zero, as one can see by noting that $\mat S_n$ is an $n$ by $n$ matrix with rank at most $\min\{m,n\}$.  
Also note that any non-zero eigenvalues for $\mat S_n$ must be strictly positive since $\mat S_n$ is non-negative definite.  

Let $\lambda_1,\lambda_2,\dots,\lambda_r$ be the $r=\min\{m,n\}$ largest eigenvalues for $\mat S_n$, and consider the polynomial
$$f(\lambda_1,\dots,\lambda_r) = \left(\prod_{1\le i\le r}\lambda_i \right)\left(\prod_{1 \le i < j \le r} (\lambda_i-\lambda_j)^2\right).$$
Note that $f(\lambda_1,\dots,\lambda_r)$ is non-zero if and only if $\lambda_1,\dots,\lambda_r$ are all distinct and non-zero.  Because $f$ is a symmetric function of $\lambda_1,\dots,\lambda_r$, we can write $f$ as a polynomial in terms of the elementary symmetric functions of $\lambda_1,\dots,\lambda_r$.  All of the elementary symmetric functions of $\lambda_1,\dots, \lambda_r$ appear as coefficients of the polynomial $$g(x)=\frac{\det(\mat S_n - x\mat I)}{x^{k}}$$ (recall that at least $k$ eigenvalues of $\mat S_n$ are equal to zero); thus, we may use the fundamental theorem of symmetric polynomials to write $f(\lambda_1,\dots,\lambda_r)= \tilde f(x_{11},\dots,x_{1n}, \dots,x_{m1},\dots,x_{mn})$, a polynomial in the entries of $\mat X_n=(x_{ij})$. 

We note that $\tilde f(x_{11},\dots,x_{mn})$ is not identically zero (for example, one can choose $\mat X_n$ to be a matrix with increasing positive integers on the main diagonal and zeros elsewhere).  One can show by induction on the number of variables that any polynomial $h(y_1,\dots,y_\ell)$ that is not identically zero satisfies $\Prob(h(y_1,\dots,y_\ell) =0 ) =0$ if $y_1,\dots,y_\ell$ are independent, absolutely continuous random variables.  This completes the proof of the first assertion of Lemma~\ref{lemma:distinct}.

We now turn to proving that any eigenvector corresponding to one of the $r$ strictly positive eigenvalues must have all non-zero coordinates.  Let $v=(v_1,\dots,v_n)$ be a unit eigenvector of $\mat S_n$ corresponding to $\lambda$, one of the strictly positive eigenvalues.  Let $X_1,\dots,X_n$ be the column vectors of $\mat X_n$.  
Because $v$ is an eigenvector, we must have that $\transp{X_\ell}\mat X_n v = \lambda v_\ell$, and thus $v_\ell = 0$ if and only if $\transp{X_\ell}\mat X_n v = 0.$  
\newcommand\hatmat[1]{\hat{\mat{#1}}}

\newcommand\BOmega{E}
Let $\hat v$ be the $(n-1)$-dimensional vector formed by deleting the $\ell$-th coordinate of $v$, and let $\hatmat X_n$ be the $m$ by $n-1$ matrix formed by deleting the $\ell$-th column of $\mat X_n$.  Note that if $v_\ell=0$, then $\hat v$ is an eigenvector for $\hatmat S_n=\transp{\hatmat X_n}\hatmat X_n$, which is independent of $X_\ell$.  Thus the probability that $v_\ell=0$ is bounded by the probability of the event 
$$\BOmega:= \left\{ \mbox{\parbox{4in}{ there exist $1 \leq \ell \leq n$ and a unit eigenvector $\hat v$ of $\hatmat S_n$ with strictly positive eigenvalue such that $\hatmat X_n \hat v$ is orthogonal to the vector $X_\ell$}} \right\}.$$ 

Furthermore,
we observe that $\hatmat X_n \hat v =0$ means that $v$ (with $\ell$-th coordinate set equal to zero) is an eigenvector for $\mat S_n$ with eigenvalue $0$, a contradiction that implies that the event $\BOmega$ only occurs when $\hatmat X_n \hat v \ne 0$. Thus $\Prob(\BOmega) = \Prob(\BOmega \mbox{ and } \hatmat X_n \hat v \ne 0)$.
Finally, we note that for any unit eigenvector $\hat v$ for $\hatmat S_n$, we have 
$\Prob(\transp{X_\ell}\hatmat X_n \hat v =0 \mbox{ and } \hatmat X_n \hat v \ne 0) =0$
 by conditioning on $\hatmat X_n \hat v$ and noting that if $\hatmat X_n \hat v$ is non-zero, then with probability 1, we have $\transp{X_\ell}\hatmat X_n \hat v \ne 0$ (recall that $X_\ell$ and $\hatmat X_n \hat v$ are independent random variables and that the entries of $X_{\ell}$ are absolutely continuous).  
An application of the union bound completes the proof.

\end{document}